\documentclass[12pt]{article}

\usepackage{amsmath}
\usepackage{amsthm}
\usepackage{amssymb}
\usepackage{amscd}
\usepackage{xspace}
\usepackage{verbatim}
\newtheorem{theorem}{Theorem}[section]
\newtheorem{corollary}{Corollary}[section]
\newtheorem{lemma}{Lemma}[section]
\newtheorem{proposition}{Proposition}[section]
\theoremstyle{definition}
\newtheorem{definition}{Definition}[section]

\theoremstyle{remark}
\newtheorem{remark}{Remark}[section]
\numberwithin{equation}{section}

\newcommand{\ov}{\overline}

\newcommand{\e}{\varepsilon}

\renewcommand{\O}{\Omega}

\renewcommand{\vec}[1]{\mathbf{#1}}
\newcommand{\field}[1]{\mathbb{#1}}

\newcommand{\R}{\field{R}}

\newcommand{\er}{\eqref}

\DeclareMathOperator{\Div}{div} \DeclareMathOperator{\dist}{dist}
\DeclareMathOperator{\supp}{supp}

\renewcommand{\O}{\Omega}

\newcommand{\f}{\varphi}
\renewcommand{\vec}[1]{\boldsymbol{#1}}
\DeclareMathOperator{\essinf}{ess\,inf}

\date{}

\begin{document}
\title{
Asymptotic behavior of the $W^{1/q,q}$-norm of mollified $BV$ functions
and applications to singular perturbation problems
} \maketitle
\begin{center}
\textsc{Arkady Poliakovsky \footnote{E-mail:
poliakov@math.bgu.ac.il}
}\\[3mm]
Department of Mathematics, Ben Gurion University of the Negev,\\
P.O.B. 653, Be'er Sheva 84105, Israel
\\[2mm]
\end{center}

\begin{abstract}
Motivated by results of Figalli and Jerison \cite{AFDJ} and
Hern\'{a}ndez \cite{FHHalp}, we prove the following formula:
\begin{equation*}
    \lim_{\e\to
        0^+}\frac{1}{|\ln{\e}|}\big\|\eta_\e*u\big\|^q_{W^{1/q,q}(\Omega)}=
    C_0\int_{J_u}\Big|u^+(x)-u^-(x)\Big|^qd\mathcal{H}^{N-1}(x),
\end{equation*}
where $\Omega\subset\R^N$ is a regular domain, $u\in BV(\Omega)\cap
L^\infty$, $q>1$ and
 $\eta_\e(z)=\e^{-N}\eta(z/\e)$ is a smooth mollifier.
In addition, we  apply the above formula
to the study of certain singular perturbation problems.
\end{abstract}

\section{Introduction}
Figalli and Jerison found in \cite{AFDJ} a relationship between the
perimeter of a set and a fractional Sobolev norm of its
characteristic function. More precisely, for the mollifying kernel
$\eta_\e(z)=\e^{-N}\eta(z/\e)$, where  $\eta(z)$ denotes
the standard Gaussian in $\R^N$,   they showed that there exist
constants $C_1>0$ and $C_2>0$ such that for every set $A\subset\R^N$ of finite
perimeter $P(A)$ we have
\begin{equation}\label{hjhjjggjjjkhkhjjhjh}
C_1P(A)\leq\liminf_{\e\to
0^+}\frac{1}{|\ln{\e}|}\big\|\eta_\e*\chi_A\big\|^2_{H^{1/2}(\R^N)}\leq\limsup_{\e\to
0^+}\frac{1}{|\ln{\e}|}\big\|\eta_\e*\chi_A\big\|^2_{H^{1/2}(\R^N)}\leq
C_2P(A),
\end{equation}
where  $\chi_A$ is the characteristic function of  $A$. More recently, Hern\'{a}ndez
improved this result in \cite{FHHalp} as follows. For $\eta_\e$ as above
 he
showed that there exist a constant $C_0>0$ such that for every $u\in
BV(\R^N)\cap L^\infty$ we have
\begin{equation}\label{hjhjjggjjjkhkhjjhjhhgghjhjh}
\lim_{\e\to
0^+}\frac{1}{|\ln{\e}|}\big\|\eta_\e*u\big\|^2_{H^{1/2}(\R^N)}=
C_0\int_{J_u}\Big|u^+(x)-u^-(x)\Big|^2d\mathcal{H}^{N-1}(x).
\end{equation}
A related result in which the same R.H.S.~as in
\er{hjhjjggjjjkhkhjjhjhhgghjhjh} appears, was obtained  in
\cite{pol}. More precisely, we showed in \cite{pol} that for every radial $\eta\in
C^\infty_c(\R^N,\R)$ there exists a constant $C=C_\eta>0$ such that for
every $u\in BV(\Omega,\R^d)\cap L^\infty$ we have
\begin{equation}\label{hjhjjggjjjkhkhjjhjhhgghjhjhljkjljl}
\lim_{\e\to 0^+}\e\big\|\eta_\e*u\big\|^2_{H^{1}(\Omega)}=
C_\eta\int_{J_u}\Big|u^+(x)-u^-(x)\Big|^2d\mathcal{H}^{N-1}(x).
\end{equation}
More recently, we showed in
\cite{PJB} yet another related result:
\begin{theorem}\label{ghgghgghjjkjkzzbvqredint}
Let $\Omega\subset\R^N$ be an open set with bounded Lipschitz
boundary and let $u\in BV(\Omega,\R^d)\cap L^\infty(\Omega,\R^d)$.
Then, for every $q>1$ we have
\begin{equation}\label{fgyufghfghjgghgjkhkkGHGHKKjjjjkjkkjkmmlmjijiluuizziihhhjbvqKKredint}
\lim\limits_{\e\to
0^+}\int\limits_\Omega\int\limits_{B_\e(x)\cap\Omega}\frac{1}{\e^N}\,\frac{\big|u(
y)-u(x)\big|^q}{|y-x|}dydx=C_N\int_{J_u}\Big|u^+(x)-u^-(x)\Big|^qd\mathcal{H}^{N-1}(x),
\end{equation}
with the dimensional constant $C_N>0$ defined by
\begin{equation}\label{fgyufghfghjgghgjkhkkGHGHKKggkhhjoozzbvqkkint}
C_N:=\frac{1}{N}\int_{S^{N-1}}|z_1|d\mathcal{H}^{N-1}(z)\,,
\end{equation}
where we denote $z:=(z_1,\ldots, z_N)\in\R^N$.
\end{theorem}

In the present paper we generalize the formula
\er{hjhjjggjjjkhkhjjhjhhgghjhjh} in several aspects:
\begin{itemize}
\item We allow a general
mollifying kernel $\eta\in W^{1,1}(\R^N,\R)$ (not only the
Gaussian as before),
\item We allow a general domain $\Omega\subset\R^N$, of certain regularity, while previous results required  $\Omega=\R^N$,
\item We treat the $W^{1/q,q}(\Omega)$-norm for any $q>1$, while previous results were restricted to the case $q=2$.
\end{itemize}
Recall that the Gagliardo seminorm $\|u\|_{W^{1/q,q}(\Omega,\R^d)}$ is
given by
\begin{equation}\label{fgyufghfghjgghgjkhkkGHGHKKokuhhhugugzzkhhbvqkkklklojijghhfRRhh}
\|u\|_{W^{1/q,q}(\Omega,\R^d)}:=\Bigg(\int_{\Omega}\bigg(\int_{\Omega}\frac{\big|u(x)-u(y)\big|^q}{|x-y|^{N+1}}dy\bigg)dx\Bigg)^{\frac{1}{q}}.
\end{equation}
Our first main result is
\begin{theorem}\label{jbnvjnnjvnvnnbhh}
Let $\Omega\subset\R^N$ be an open set and let $u\in
BV(\R^N,\R^d)\cap L^\infty(\R^N,\R^d)$ be such that
$\|Du\|(\partial\Omega)=0$.
For $\eta\in W^{1,1}(\R^N,\R)$,
every $x\in\R^N$ and every $\e>0$ define
\begin{equation}\label{jmvnvnbccbvhjhjhhjjkhgjgGHKKzzbvqkhhihkinthh}
u_\e(x):=\frac{1}{\e^N}\int_{\R^N}\eta\Big(\frac{y-x}{\e}\Big)u(y)dy=(\eta_\e*u)(x).
\end{equation}
Then, for any $q>1$ we have
\begin{multline}\label{fgyufghfghjgghgjkhkkGHGHKKokuhhjujkjhjhhhhugugzzkhhbvqkkklkljjkjkjhjgklkljkll;k;k;ouuiojkkjhkhkljl;klk;k;jhkhkhjjkhkjljjk;kk;jkkhihkjljllmjhnjghghgjhjhjhhjhjRRhh}
\lim_{\e\to
0^+}\frac{1}{|\ln{\e}|}\|u_\e\|^q_{W^{1/q,q}(\Omega,\R^d)}=
\\
2\bigg|\int_{\R^N}\eta(z)dz\bigg|^q\Bigg(\int_{\R^{N-1}}\frac{dv}{\big(\sqrt{1+|v|^2}\big)^{N+1}}\Bigg)\int_{J_u\cap\Omega}
\Big|u^+(x)-u^-(x)\Big|^q d\mathcal{H}^{N-1}(x).
\end{multline}
\end{theorem}
Theorem \ref{jbnvjnnjvnvnnbhh} enables us to prove an upper bound, in
the limit $\e\to0^+$, for the following
singular perturbation functionals with differential constraints:
\begin{itemize}
\item[({\bf i})]
\begin{equation}\label{fgyufghfghjgghgjkhkkGHGHKKokuhhjujkjhjhhhhugugzzkhhbvqkkklkljjkjkjhjgklkljkll;k;k;ouuiojkkjhkhkljl;klk;k;jhkhkhjjkhkjljjk;kk;jkkhihkjljllmjhnjghghgjhjhjhhjhjRRRRkkjjlkljjljljj}
E^{(1)}_\e(v):=\begin{cases}\frac{1}{|\ln{\e}|}\|v\|^q_{W^{1/q,q}(\Omega,\R^d)}+\frac{1}{\e}\int_{\Omega}W\big(v,x\big)dx\quad\quad\text{if}\quad A\cdot\nabla v=0\\
+\infty\quad\quad\text{otherwise},
\end{cases}
\end{equation}
for $v:\Omega\to\R^d$;
\item[({\bf ii})]
\begin{equation}\label{fgyufghfghjgghgjkhkkGHGHKKokuhhjujkjhjhhhhugugzzkhhbvqkkklkljjkjkjhjgklkljkll;k;k;ouuiojkkjhkhkljl;klk;k;jhkhkhjjkhkjljjk;kk;jkkhihkjljllmjhnjghghgjhjhjhhjhjRRRRkkjjlkljjljljjhjjggj}
E^{(2)}_\e(v):=\begin{cases}\frac{1}{|\ln{\e}|}\left(\|v\|^q_{W^{1/q,q}(\R^N,\R^d)}-\|v\|^q_{W^{1/q,q}(\R^N\setminus\ov \Omega,\R^d)}\right)+\frac{1}{\e}\int_{\Omega}W\big(v,x\big)dx\quad\text{if}\quad A\cdot\nabla v=0\\
+\infty\quad\quad\text{otherwise},
\end{cases}
\end{equation}
for $v:\R^N\to\R^d$.
\end{itemize}
In both cases $A:\R^{d\times N}\to\R^l$ is a linear operator
(possibly trivial). The most important particular cases are the
following:
\begin{itemize}
\item[{(\bf a)}] $A\equiv 0$ (i.e., without any prescribed differential
constraint),
\item[{(\bf b)}] $d=N$, $l=N^2$ and $A\cdot\nabla v\equiv \text{curl}\, v:=\big\{\partial_kv_j-\partial_jv_k\big\}_{1\leq k,j\leq N}$,
\item[{(\bf c)}] $l=d$ and $A\cdot\nabla v\equiv \Div v$.
\end{itemize}
The $\Gamma$-limit of the functional
\er{fgyufghfghjgghgjkhkkGHGHKKokuhhjujkjhjhhhhugugzzkhhbvqkkklkljjkjkjhjgklkljkll;k;k;ouuiojkkjhkhkljl;klk;k;jhkhkhjjkhkjljjk;kk;jkkhihkjljllmjhnjghghgjhjhjhhjhjRRRRkkjjlkljjljljj}
in the $L^p$-topology when  $A\equiv 0$, $q=2$, $N=1$  and $W$ is a
double-well potential was found by Alberti, Bouchitt\'e and Seppecher~\cite{GaGbPs}. The result was
generalized to any dimension $N\ge1$, for
 the functional
\er{fgyufghfghjgghgjkhkkGHGHKKokuhhjujkjhjhhhhugugzzkhhbvqkkklkljjkjkjhjgklkljkll;k;k;ouuiojkkjhkhkljl;klk;k;jhkhkhjjkhkjljjk;kk;jkkhihkjljllmjhnjghghgjhjhjhhjhjRRRRkkjjlkljjljljjhjjggj},
by Savin and Valdinoci~\cite{OSEV}.
\par  Note that the
functional
\er{fgyufghfghjgghgjkhkkGHGHKKokuhhjujkjhjhhhhugugzzkhhbvqkkklkljjkjkjhjgklkljkll;k;k;ouuiojkkjhkhkljl;klk;k;jhkhkhjjkhkjljjk;kk;jkkhihkjljllmjhnjghghgjhjhjhhjhjRRRRkkjjlkljjljljj}
resembles the energy functional in the following singular
perturbation problem:
\begin{equation}\label{fgyufghfghjgghgjkhkkGHGHKKokuhhjujkjhjhhhhugugzzkhhbvqkkklkljjkjkjhjgklkljkll;k;k;ouuiojkkjhkhkljl;klk;k;jhkhkhjjkhkjljjk;kk;jkkhihkjljllmjhnjghghgjhjhjhhjhjRRRRkkjjlkljjljljjjhhhj}
\hat E_\e(v):=\begin{cases}\e^{q-1}\|v\|^q_{W^{1,q}(\Omega,\R^d)}+\frac{1}{\e}\int_{\Omega}W\big(v,x\big)dx\quad\quad\text{if}\quad A\cdot\nabla v=0\\
+\infty\quad\quad\text{otherwise},
\end{cases}
\end{equation}
that attracted  a lot of attention by many authors, starting from Modica and
Mortola \cite{mm1}, Modica \cite{modica}, Sternberg
\cite{sternberg} and others, who studied the basic special case of
\er{fgyufghfghjgghgjkhkkGHGHKKokuhhjujkjhjhhhhugugzzkhhbvqkkklkljjkjkjhjgklkljkll;k;k;ouuiojkkjhkhkljl;klk;k;jhkhkhjjkhkjljjk;kk;jkkhihkjljllmjhnjghghgjhjhjhhjhjRRRRkkjjlkljjljljjjhhhj}
with $A\equiv 0$, $q=2$ and  $W$ being a double-well potential.
The $\Gamma$ limit of
\er{fgyufghfghjgghgjkhkkGHGHKKokuhhjujkjhjhhhhugugzzkhhbvqkkklkljjkjkjhjgklkljkll;k;k;ouuiojkkjhkhkljl;klk;k;jhkhkhjjkhkjljjk;kk;jkkhihkjljllmjhnjghghgjhjhjhhjhjRRRRkkjjlkljjljljjjhhhj}
with $A\equiv 0$, $q=2$ and a general $W\in C^0$ that does not depend on $x$, was found by Ambrosio in \cite{ambrosio}.
As an example with a nontrivial differential constraint we mention the Aviles-Giga
functional, that  appear in various applications. It is defined for scalar functions $\psi$ by
\begin{equation}
\label{b5..jkhh} \tilde
E_\varepsilon(\psi):=\int_\Omega\left\{\varepsilon|\nabla^2
\psi|^2+\frac{1}{\varepsilon}\Big(1-|\nabla
\psi|^2\Big)^2\right\}dx\quad
\quad\text{(see \cite{adm,ag1,ag2})},
\end{equation}
and the objective is to study the
$\Gamma$-limit, as $\varepsilon\to 0^+$. This can be seen as
 a special  case of
\er{fgyufghfghjgghgjkhkkGHGHKKokuhhjujkjhjhhhhugugzzkhhbvqkkklkljjkjkjhjgklkljkll;k;k;ouuiojkkjhkhkljl;klk;k;jhkhkhjjkhkjljjk;kk;jkkhihkjljllmjhnjghghgjhjhjhhjhjRRRRkkjjlkljjljljjjhhhj}
if we set $v:=\nabla\psi$ and let  $A\cdot \nabla v\equiv \text{curl}\, v$, $q=2$ and
$W(v,x)=(1-|v|^2)^2$.

Our second result provides an upper bound for the energies
\er{fgyufghfghjgghgjkhkkGHGHKKokuhhjujkjhjhhhhugugzzkhhbvqkkklkljjkjkjhjgklkljkll;k;k;ouuiojkkjhkhkljl;klk;k;jhkhkhjjkhkjljjk;kk;jkkhihkjljllmjhnjghghgjhjhjhhjhjRRRRkkjjlkljjljljj}-\er{fgyufghfghjgghgjkhkkGHGHKKokuhhjujkjhjhhhhugugzzkhhbvqkkklkljjkjkjhjgklkljkll;k;k;ouuiojkkjhkhkljl;klk;k;jhkhkhjjkhkjljjk;kk;jkkhihkjljllmjhnjghghgjhjhjhhjhjRRRRkkjjlkljjljljjhjjggj}:
\begin{theorem}\label{jbnvjnnjvnvnnbRRkkjj}
Let  $\Omega\subset\R^N$ be an open set and let
$W:\R^d\times\R^N\to\R$ be a Borel measurable nonnegative
function, continuous and continuously differentiable w.r.t.~the first argument, such that $W(0,\cdot)\in L^1(\Omega,\R)$. Assume further  that for every
$D>0$ there exists $C:=C_D>0$ such that
\begin{equation}\label{jmvnvnbccbvhjhjhhjjkhgjgGHKKzzbvqkhhihkRRruuyrfhjkkhhhhhjkk}
\big|\nabla_b W(b,x)\big|\leq C_D\quad\quad\forall
x\in\R^N,\;\forall\, b\in B_D(0).
\end{equation}
Let $u\in BV(\R^N,\R^d)\cap L^\infty(\R^N,\R^d)$ be
such that  $W\big(u(x),x\big)=0$ a.e. in $\Omega$,
$\|Du\|(\partial\Omega)=0$, and $A\cdot Du=0$ in $\R^N$, where
$A:\R^{d\times N}\to\R^l$ is a prescribed linear operator (possibly
trivial).
Then, for any $q>1$ there exists a sequence of functions
$\big\{\psi_\e\big\}_{\e>0}\subset C^\infty(\R^N,\R^d)\cap
W^{1,1}(\R^N,\R^d)\cap W^{1,\infty}(\R^N,\R^d)$
such that $A\cdot D\psi_\e=0$ in $\R^N$, $\psi_\e(x)\to u(x)$
strongly in $L^p(\R^N,\R^d)$ for every $p\geq 1$, and
\begin{multline}\label{fgyufghfghjgghgjkhkkGHGHKKokuhhjujkjhjhhhhugugzzkhhbvqkkklkljjkjkjhjgklkljkll;k;k;ouuiojkkjhkhkljl;klk;k;jhkhkhjjkhkjljjk;kk;jkkhihkjljllmjhnjghghgjhjhjhhjhjRRRRkk}
\limsup_{\e\to
0^+}\bigg(\frac{1}{|\ln{\e}|}\Big(\|\psi_\e\|^q_{W^{1/q,q}(\R^N,\R^d)}-\|\psi_\e\|^q_{W^{1/q,q}(\R^N\setminus
\ov\Omega,\R^d)}\Big)+\frac{1}{\e}\int_{\Omega}W\Big(\psi_\e(x),x\Big)dx
\bigg)=\\
\limsup_{\e\to
0^+}\bigg(\frac{1}{|\ln{\e}|}\|\psi_\e\|^q_{W^{1/q,q}(\Omega,\R^d)}+\frac{1}{\e}\int_{\Omega}W\Big(\psi_\e(x),x\Big)dx
\bigg)=\\
\Bigg(\int_{\R^{N-1}}\frac{2}{\big(\sqrt{1+|v|^2}\big)^{N+1}}dv\Bigg)\int_{J_u\cap\Omega}
\big|u^+(y)-u^-(y)\big|^q d\mathcal{H}^{N-1}(y).
\end{multline}
Moreover, in the case $A\equiv 0$ we can   choose $\psi_\e$ to
satisfy also
\begin{equation}\label{jmvnvnbccbvhjhjhhjjkhgjgGHKKzzbvqkhhihkRRkkhihhjhjhjhkk}
\int_\Omega \psi_\e(x)dx=\int_\Omega u(x)dx\quad\quad\forall\e>0.
\end{equation}
\end{theorem}
Unfortunately, the upper bound found in Theorem
\ref{jbnvjnnjvnvnnbRRkkjj} is not sharp in the most general case
with a nontrivial prescribed differential constraint. For example, in
the particular case of
\er{fgyufghfghjgghgjkhkkGHGHKKokuhhjujkjhjhhhhugugzzkhhbvqkkklkljjkjkjhjgklkljkll;k;k;ouuiojkkjhkhkljl;klk;k;jhkhkhjjkhkjljjk;kk;jkkhihkjljllmjhnjghghgjhjhjhhjhjRRRRkkjjlkljjljljj}
with $N=2$, $A\cdot \nabla v\equiv \text{curl}\, v$, $q>3$ and
$W(v,x)=(1-|v|^2)^2$, the functional on the R.H.S.~of \er{fgyufghfghjgghgjkhkkGHGHKKokuhhjujkjhjhhhhugugzzkhhbvqkkklkljjkjkjhjgklkljkll;k;k;ouuiojkkjhkhkljl;klk;k;jhkhkhjjkhkjljjk;kk;jkkhihkjljllmjhnjghghgjhjhjhhjhjRRRRkk}
 is not lower semicontinuous, hence  cannot be the
 $\Gamma$-limit (see \cite{adm}).
 However, we
still hope that the result of the above theorem could provide the
sharp upper bound in some cases with $A=0$. Indeed,  the
$\Gamma$-limit, computed in \cite{GaGbPs} for  the special case of
\er{fgyufghfghjgghgjkhkkGHGHKKokuhhjujkjhjhhhhugugzzkhhbvqkkklkljjkjkjhjgklkljkll;k;k;ouuiojkkjhkhkljl;klk;k;jhkhkhjjkhkjljjk;kk;jkkhihkjljllmjhnjghghgjhjhjhhjhjRRRRkkjjlkljjljljj}
 with  $A\equiv 0$, $q=2$, $N=1$ and $W$
being a double well potential,
coincides
 with the upper bound found in Theorem
\ref{jbnvjnnjvnvnnbRRkkjj}. Moreover, since the functional in
\er{fgyufghfghjgghgjkhkkGHGHKKokuhhjujkjhjhhhhugugzzkhhbvqkkklkljjkjkjhjgklkljkll;k;k;ouuiojkkjhkhkljl;klk;k;jhkhkhjjkhkjljjk;kk;jkkhihkjljllmjhnjghghgjhjhjhhjhjRRRRkkjjlkljjljljjhjjggj}
is superior to the functional in
\er{fgyufghfghjgghgjkhkkGHGHKKokuhhjujkjhjhhhhugugzzkhhbvqkkklkljjkjkjhjgklkljkll;k;k;ouuiojkkjhkhkljl;klk;k;jhkhkhjjkhkjljjk;kk;jkkhihkjljllmjhnjghghgjhjhjhhjhjRRRRkkjjlkljjljljj},
the $\Gamma$-limit, found in \cite{OSEV} (see also \cite{GPOSEV})
for the energy
\er{fgyufghfghjgghgjkhkkGHGHKKokuhhjujkjhjhhhhugugzzkhhbvqkkklkljjkjkjhjgklkljkll;k;k;ouuiojkkjhkhkljl;klk;k;jhkhkhjjkhkjljjk;kk;jkkhihkjljllmjhnjghghgjhjhjhhjhjRRRRkkjjlkljjljljjhjjggj}
in any dimension $N\geq 1$ with $A\equiv 0$, $q=2$ and $W$ being a
double well potential, coincides again with our upper bound.

The paper is organized as follows. In section \ref{hjggjggfgfh} we
prove  our two main results. For the convenience of the reader, in
the Appendix
we recall some known results on $BV$ functions, needed for the
proofs.

\section{Proof of the main results}\label{hjggjggfgfh}
\begin{proposition}\label{jbnvjnnjvnvnnb}
Let $q>1$, $\Omega\subset\R^N$ be an open set and $u\in
BV(\R^N,\R^d)\cap L^\infty(\R^N,\R^d)$ be such that
$\|Du\|(\partial\Omega)=0$.
Let $\eta\in C^\infty_c(\R^N,\R)$ and
for every $x\in\R^N$ and every $\e>0$ define
\begin{equation}\label{jmvnvnbccbvhjhjhhjjkhgjgGHKKzzbvqkhhihkint}
u_\e(x):=\frac{1}{\e^N}\int_{\R^N}\eta\Big(\frac{y-x}{\e}\Big)u(y)dy=(\eta_\e*u)(x).
\end{equation}
Then,
\begin{multline}\label{fgyufghfghjgghgjkhkkGHGHKKokuhhjujkjhjhhhhugugzzkhhbvqkkklkljjkjkjhjgklkljkll;k;k;ouuiojkkjhkhkljl;klk;k;jhkhkhjjkhkjljjk;kk;jkkhihkjljllmjhnjghghgjhjhjhhjhjRR}
\lim_{\e\to
0^+}\frac{1}{|\ln{\e}|}\|u_\e\|^q_{W^{1/q,q}(\Omega,\R^d)}
=
\\
2\bigg|\int_{\R^N}\eta(z)dz\bigg|^q\Bigg(\int_{\R^{N-1}}\frac{1}{\big(\sqrt{1+|v|^2}\big)^{N+1}}dv\Bigg)\int_{J_u\cap\Omega}
\Big|u^+(x)-u^-(x)\Big|^q d\mathcal{H}^{N-1}(x).
\end{multline}
\end{proposition}
\begin{proof}
We start with some notations. For every $\vec\nu\in S^{N-1}$ and
$x\in\R^N$
set
\begin{align}
H_+(x,\vec\nu)&=\{\xi\in\R^N\,:\,(\xi-x)\cdot\vec\nu>0\}\,,\label{HN+}\\
H_-(x,\vec\nu)&=\{\xi\in\R^N\,:\,(\xi-x)\cdot\vec\nu<0\}\label{HN-}\\
\intertext{and} H_0(\vec \nu)&=\{ \xi\in\R^N\,:\, \xi\cdot\vec
\nu=0\}\label{HN}\,.
\end{align}
%
%
%
%
Let $R>0$ be such that
$\supp\eta\subset B_R(0)$. For every $x\in\R^N$ and every
$\e>0$ we rewrite \er{jmvnvnbccbvhjhjhhjjkhgjgGHKKzzbvqkhhihkint}
as:
\begin{equation}\label{jmvnvnbccbvhjhjhhjjkhgjgGHKKzzbvqkhhihk}
u_\e(x):=\frac{1}{\e^N}\int_{\R^N}\eta\Big(\frac{y-x}{\e}\Big)u(y)dy=\int_{\R^N}\eta(z)u(x+\e
z)dz=\int_{B_R(0)}\eta(z)u(x+\e z)dz.
\end{equation}
By \er{jmvnvnbccbvhjhjhhjjkhgjgGHKKzzbvqkhhihk}
we have
\begin{multline}\label{jmvnvnbccbvhjhjhhjjkhgjgGHKKzzbvqkhhihkhkhhj}
\frac{d}{d\e}u_\e(x):=-\frac{N}{\e^{N+1}}\int_{\R^N}\eta\Big(\frac{y-x}{\e}\Big)u(y)dy
-\frac{1}{\e^{N}}\int_{\R^N}\frac{y-x}{\e^2}\cdot\nabla\eta\Big(\frac{y-x}{\e}\Big)u(y)dy
=\\-\frac{1}{\e^{N}}\int_{\R^N}\Div_{y}\bigg\{\eta\Big(\frac{y-x}{\e}\Big)\frac{y-x}{\e}\bigg\}u(y)dy=
\frac{1}{\e^{N}}\int_{\R^N}\eta\Big(\frac{y-x}{\e}\Big)\frac{y-x}{\e}\cdot
d\big[Du(y)\big].
\end{multline}
Moreover, by \eqref{fgyufghfghjgghgjkhkkGHGHKKokuhhhugugzzkhhbvqkkklklojijghhfRRhh}
 we have
\begin{multline}\label{fgyufghfghjgghgjkhkkGHGHKKokuhhhugugzzkhhbvqkkklklojijghhf}
\|u_\e\|^q_{W^{1/q,q}}=\|u_\e\|^q_{W^{1/q,q}(\Omega,\R^d)}=\int_{\R^N}\bigg(\int_{\R^N}\frac{\big|u_\e(x)-u_\e(y)\big|^q}{|x-y|^{N+1}}\chi_\Omega(y)dy\bigg)\chi_\Omega(x)dx
\\=\int_{\R^N}\bigg(\int_{\R^N}\frac{\big|u_\e(x+z)-u_\e(x)\big|^q}{|z|^{N+1}}\chi_\Omega(x+z)\chi_\Omega(x)dz\bigg)dx,
\end{multline}
where
\begin{equation}\label{fhjgjgfgjjk}
\chi_\Omega(x):=\begin{cases} 1\quad\forall x\in\Omega
\\
0\quad\forall x\in\R^N\setminus\Omega
\end{cases}.
\end{equation}
Thus,
\begin{equation}\label{fgyufghfghjgghgjkhkkGHGHKKokuhhhugugzzkhhbvqkkklkl}
\frac{1}{-\ln{\e}}\|u_\e\|^q_{W^{1/q,q}}=-\frac{1}{\ln{\e}}\int_{\R^N}\bigg(\int_{\R^N}\frac{\big|u_\e(x+z)-u_\e(x)\big|^q}{|z|^{N+1}}
\chi_\Omega(x+z)\chi_\Omega(x)dz\bigg)dx.
\end{equation}
Since $-\ln{\e}\to+\infty$ as $\e\to 0^+$, applying L'H\^{o}pital's
rule to the expression in
\er{fgyufghfghjgghgjkhkkGHGHKKokuhhhugugzzkhhbvqkkklkl} yields
\begin{multline}\label{fgyufghfghjgghgjkhkkGHGHKKokuhhhugugzzkhhbvqkkklkljjkjk}
\lim_{\e\to
0^+}\frac{1}{-\ln{\e}}\|u_\e\|^q_{W^{1/q,q}}=\\-\lim\limits_{\e\to
0^+}\int\limits_{\R^N}\Bigg(\int\limits_{\R^N}\frac{\e}{|z|^{N+1}}\bigg(\frac{d}{d\e}\Big(u_\e(x+z)-u_\e(x)\Big)\bigg)\cdot\nabla
F_q\big(u_\e(x+z)-u_\e(x)\big)\chi_\Omega(x+z)\chi_\Omega(x)dz\Bigg)dx,
\end{multline}
where $F_q\in C^1(\R^d,\R)$ is defined by
\begin{equation}\label{fhjgjgfgjjkjhhjhj}
F_q(h):=|h|^q \quad\quad\forall h\in\R^d.
\end{equation}
Thus, by
\er{fgyufghfghjgghgjkhkkGHGHKKokuhhhugugzzkhhbvqkkklkljjkjk},
\er{jmvnvnbccbvhjhjhhjjkhgjgGHKKzzbvqkhhihk} and
\er{jmvnvnbccbvhjhjhhjjkhgjgGHKKzzbvqkhhihkhkhhj} we get
\begin{multline}\label{11fgyufghfghjgghgjkhkkGHGHKKokuhhhugugzzkhhbvqkkklkljjkjkklkljkll;k;k;}
\lim_{\e\to
0^+}\frac{1}{-\ln{\e}}\|u_\e\|^q_{W^{1/q,q}}=\\-\lim_{\e\to
0^+}\int_{\R^N}\int_{\R^N}\frac{\e}{|z|^{N+1}}\Bigg\{\frac{1}{\e^{N}}\int_{\R^N}\bigg(\eta\Big(\frac{y-(x+z)}{\e}\Big)\frac{y-(x+z)}{\e}-
\eta\Big(\frac{y-x}{\e}\Big)\frac{y-x}{\e}\bigg)\cdot
d\big[Du(y)\big]\Bigg\}\times\\
\times\nabla
F_q\Bigg(\int_{\R^N}\eta(\xi)\Big(u(x+z+\e\xi)-u(x+\e\xi)\Big)
d\xi\Bigg)\chi_\Omega(x+z)\chi_\Omega(x) dzdx =\\-\lim_{\e\to
0^+}\int_{\R^N}\int_{\R^N}\int_{\R^N}\frac{\e}{|z|^{N+1}}\frac{1}{\e^{N}}\bigg(\eta\Big(\frac{y-(x+z)}{\e}\Big)\frac{y-(x+z)}{\e}-
\eta\Big(\frac{y-x}{\e}\Big)\frac{y-x}{\e}\bigg)\times\\
\times\nabla
F_q\Bigg(\int_{\R^N}\eta(\xi)\Big(u(x+z+\e\xi)-u(x+\e\xi)\Big)
d\xi\Bigg)\chi_\Omega(x+z)\chi_\Omega(x) dzdx\cdot d\big[Du(y)\big].
\end{multline}
Changing variable, $z/\e\to z$, in the integration on the
R.H.S.~of
\er{11fgyufghfghjgghgjkhkkGHGHKKokuhhhugugzzkhhbvqkkklkljjkjkklkljkll;k;k;}
gives
\begin{multline}\label{fgyufghfghjgghgjkhkkGHGHKKokuhhhugugzzkhhbvqkkklkljjkjkklkljkll;k;k;}
\lim_{\e\to
0^+}\frac{1}{-\ln{\e}}\|u_\e\|^q_{W^{1/q,q}}=\\-\lim_{\e\to
0^+}\int_{\R^N}\int_{\R^N}\int_{\R^N}\frac{1}{|z|^{N+1}}\frac{1}{\e^{N}}\bigg(\eta\Big(\frac{y-x}{\e}-z\Big)\Big(\frac{y-x}{\e}-z\Big)-
\eta\Big(\frac{y-x}{\e}\Big)\frac{y-x}{\e}\bigg)\times\\
\times\nabla F_q\Bigg(\int_{\R^N}\eta(\xi)\Big(u(x+\e
z+\e\xi)-u(x+\e\xi)\Big) d\xi\Bigg)\chi_\Omega(x+\e
z)\chi_\Omega(x)dzdx\cdot d\big[Du(y)\big] =\\-\lim_{\e\to
0^+}\int_{\R^N}\int_{\R^N}\int_{\R^N}\frac{1}{|z|^{N+1}}\bigg(\eta\big(x-z\big)\big(x-z\big)-
\eta\big(x\big)x\bigg)\times\\ \times\nabla
F_q\Bigg(\int_{\R^N}\eta(\xi)\Big(u(y+\e z+\e\xi-\e x)-u(y+\e\xi-\e
x)\Big) d\xi\Bigg)\chi_\Omega(y-\e x+\e z)\chi_\Omega(y-\e x)
dzdx\cdot d\big[Du(y)\big].
\end{multline}
Therefore,
\begin{multline}\label{fgyufghfghjgghgjkhkkGHGHKKokuhhhugugzzkhhbvqkkklkljjkjkklkljkll;k;k;ouuio}
\lim_{\e\to
0^+}\frac{1}{-\ln{\e}}\|u_\e\|^q_{W^{1/q,q}}\\=-\lim_{\e\to
0^+}\int_{\R^N}\int_{\R^N}\int_{\R^N}\frac{1}{|z|^{N+1}}\bigg(\eta\big(x-z\big)\big(x-z\big)-
\eta\big(x\big)x\bigg)\times\\
\times\nabla
F_q\Bigg(\int_{\R^N}\Big(\eta(\xi-z)-\eta(\xi)\Big)u(y+\e\xi-\e
x)d\xi\Bigg)\chi_\Omega(y-\e x+\e z)\chi_\Omega(y-\e x)dzdx\cdot
d\big[Du(y)\big]\\
=-\lim_{\e\to
0^+}\int_{\R^N}\int_{\R^N}\int_{\R^N}\frac{1}{|z|^{N+1}}\bigg(\eta\big(x-z\big)\big(x-z\big)-
\eta\big(x\big)x\bigg)\times\\
\times\nabla
F_q\Bigg(\int_{\R^N}\Big(\eta(\xi+x-z)-\eta(\xi+x)\Big)u(y+\e\xi)d\xi\Bigg)\chi_\Omega(y-\e
x+\e z)\chi_\Omega(y-\e x) dzdx\cdot d\big[Du(y)\big].
\end{multline}
On the other hand, by \er{gjgggghfhf} in the Appendix, for every $x,z\in\R^N$ and
$\mathcal{H}^{N-1}$-a.e. $y\in\R^N$ we have
\begin{multline}\label{fgyufghfghjgghgjkhkkGHGHKKokuhhhugugzzkhhbvqkkklkljjkjkklkljkll;k;k;ouuiojkkjhkhkljljhmhmhjj}
\lim_{\e\to
0^+}\Bigg\{\int_{\R^N}\Big(\eta(\xi+x-z)-\eta(\xi+x)\Big)u(y+\e\xi)d\xi\Bigg\}=\\
u^+(y)\int_{H_+(0,\vec\nu(y))
}\Big(\eta(\xi+x-z)-\eta(\xi+x)\Big)d\xi+u^-(y)\int_{H_-(0,\vec\nu(y))}\Big(\eta(\xi+x-z)-\eta(\xi+x)\Big)d\xi.
\end{multline}
with $H_\pm(x,\vec\nu)$ as defined in \er{HN+} and \er{HN-}. Thus,
since $\|Du\|(\partial\Omega)=0$, by
\er{fgyufghfghjgghgjkhkkGHGHKKokuhhhugugzzkhhbvqkkklkljjkjkklkljkll;k;k;ouuiojkkjhkhkljljhmhmhjj}
and the Dominated Convergence Theorem we obtain:
\begin{multline}\label{11fgyufghfghjgghgjkhkkGHGHKKokuhhhugugzzkhhbvqkkklkljjkjkklkljkll;k;k;ouuiojkkjhkhkljl;klk;k;jhkh}
\lim_{\e\to 0^+}\frac{1}{-\ln{\e}}\|u_\e\|^q_{W^{1/q,q}}
=\\-\int\limits_{\R^N}\int\limits_{\R^N}\int\limits_{\R^N}\frac{1}{|z|^{N+1}}\bigg(\eta\big(x-z\big)\big(x-z\big)-
\eta\big(x\big)x\bigg)
\nabla
F_q\Bigg(u^+(y)\int\limits_{H_+(0,\vec\nu(y))}\Big(\eta(\xi+x-z)-\eta(\xi+x)\Big)d\xi\\+u^-(y)\int\limits_{H_-(0,\vec\nu(y))}\Big(\eta(\xi+x-z)-\eta(\xi+x)\Big)d\xi\Bigg)
\chi^2_\Omega(y) dzdx\cdot d\big[Du(y)\big]=\\
-\int\limits_{\Omega}\int\limits_{\R^N}\int\limits_{\R^N}\frac{1}{|z|^{N+1}}\bigg(\eta\big(x-z\big)\big(x-z\big)-
\eta\big(x\big)x\bigg)  \nabla
F_q\Bigg(u^+(y)\int\limits_{H_+(0,\vec\nu(y))}\Big(\eta(\xi+x-z)-\eta(\xi+x)\Big)d\xi\\+u^-(y)\int\limits_{H_-(0,\vec\nu(y))}\Big(\eta(\xi+x-z)-\eta(\xi+x)\Big)d\xi\Bigg)
dzdx\cdot d\big[Du(y)\big].\\
\end{multline}
It follows that
\begin{multline}\label{huuuuuuuuuuuuuuuuuuuuuuuuuuuuuuuuuufgyfgf}
\lim_{\e\to 0^+}\frac{1}{-\ln{\e}}\|u_\e\|^q_{W^{1/q,q}}=-\int_{\Omega}\int_{\R^N}\int_{\R^N}\frac{1}{|z|^{N+1}}\bigg(\eta\big(x-z\big)\big(x-z\big)-
\eta\big(x\big)x\bigg) \times\\ \times \nabla
F_q\Bigg(\big(u^+(y)-u^-(y)\big)\int_{H_+(0,\vec\nu(y))}\Big(\eta(\xi+x-z)-\eta(\xi+x)\Big)d\xi\\+u^-(y)\int_{\R^N}\Big(\eta(\xi+x-z)-\eta(\xi+x)\Big)d\xi\Bigg)
dzdx\cdot d\big[Du(y)\big]\\
=-\int_{\Omega}\int_{\R^N}\int_{\R^N}\frac{1}{|z|^{N+1}}\bigg(\eta\big(x-z\big)\big(x-z\big)-
\eta\big(x\big)x\bigg) \times\\ \times \nabla
F_q\Bigg(\big(u^+(y)-u^-(y)\big)\int_{H_+(0,\vec\nu(y))}\Big(\eta(\xi+x-z)-\eta(\xi+x)\Big)d\xi\Bigg)dzdx\cdot
d\big[Du(y)\big],
\end{multline}
where we used in the last step the fact that
$\int_{\R^N}\eta(\xi+x-z)d\xi=\int_{\R^N}\eta(\xi+x)d\xi$. Next, by
\er{huuuuuuuuuuuuuuuuuuuuuuuuuuuuuuuuuufgyfgf} and
\er{fhjgjgfgjjkjhhjhj} we infer that
\begin{multline}\label{fgyufghfghjgghgjkhkkGHGHKKokuhhhugugzzkhhbvqkkklkljjkjkklkljkll;k;k;ouuiojkkjhkhkljl;klk;k;jhkhkhkhkjljjk;kk;}
\lim_{\e\to 0^+}\frac{1}{-\ln{\e}}\|u_\e\|^q_{W^{1/q,q}} =
-\int_{\Omega}\int_{\R^N}\int_{\R^N}\frac{1}{|z|^{N+1}}\bigg(\eta\big(x-z\big)\big(x-z\big)-
\eta\big(x\big)x\bigg) \times\\ \times \nabla
F_q\Bigg(\big(u^+(y)-u^-(y)\big)\bigg(\int_{
H_+(x-z,\vec\nu(y))}\eta(\xi)d\xi-\int_{H_+(x,\vec\nu(y))
}\eta(\xi)d\xi\bigg) \Bigg)dzdx\cdot d\big[Du(y)\big]
\\=
\int_{J_u\cap\Omega}\int_{\R^N}\int_{\R^N}\frac{1}{|z|^{N+1}}\bigg(\eta\big(x\big)x\cdot\vec
\nu(y)-\eta\big(x-z\big)\big(x-z\big)\cdot\vec
\nu(y)\bigg) \times\\
\times \frac{dG_q}{d\rho}\Bigg(\int_{(x-z)\cdot\vec
\nu(y)}^{x\cdot\vec \nu(y)}\int_{H_0(\vec\nu(y))
}\eta(t\vec\nu(y)+\xi)d\mathcal{H}^{N-1}(\xi)dt \Bigg)
dxdz\big|u^+(y)-u^-(y)\big|^q d\mathcal{H}^{N-1}(y),
\end{multline}
where $G_q(\rho)\in C^1(\R,\R)$ is defined by
\begin{equation}\label{fhjgjgfgjjkjhhjhjhhjh}
G_q(\rho):=|\rho|^q \quad\quad\forall \rho\in\R,
\end{equation}
and $H_0(\vec\nu)$ is defined in \er{HN}. Therefore,
\begin{multline}\label{fgyufghfghjgghgjkhkkGHGHKKokuhhhugugzzkhhbvqkkklkljjkjkklkljkll;k;k;ouuiojkkjhkhkljl;klk;k;jhkhkhkhkjljjk;kk;jkkhihk}
\lim_{\e\to 0^+}\frac{1}{-\ln{\e}}\|u_\e\|^q_{W^{1/q,q}} =\\
\int_{J_u\cap\Omega}\int_{\R^N}\int_{\R}\int_{H_0(\vec\nu(y))}\frac{1}{|z|^{N+1}}\Bigg(
\eta\big(s\vec\nu(y)+\zeta\big)s
-\eta\Big(\big(s-z\cdot\vec\nu(y)\big)\vec\nu(y)+\zeta\Big)\big(s-z\cdot\vec
\nu(y)\big)\Bigg) \times\\
\times \frac{dG_q}{d\rho}\Bigg(\int_{s-z\cdot\vec
\nu(y)}^{s}\int_{H_0(\vec\nu(y))}\eta(t\vec\nu(y)+\xi)d\mathcal{H}^{N-1}(\xi)dt
\Bigg)
d\mathcal{H}^{N-1}(\zeta)dsdz
\big|u^+(y)-u^-(y)\big|^q d\mathcal{H}^{N-1}(y) \\
=
\int_{J_u\cap\Omega}\Bigg(\int_{\R^{N-1}}\int_{\R}\int_{\R}\frac{1}{\big(\sqrt{\tau^2+|w|^2}\big)^{N+1}}\times\\
\times\Bigg( \int_{H_0(\vec\nu(y))}\bigg(
\eta\big(s\vec\nu(y)+\zeta\big)s
-\eta\Big(\big(s-\tau\big)\vec\nu(y)+\zeta\Big)\big(s-\tau\big)\bigg)d\mathcal{H}^{N-1}(\zeta)\Bigg) \times\\
\times
\frac{dG_q}{d\rho}\bigg(\int_{s-\tau}^{s}\int_{H_0(\vec\nu(y))}\eta(t\vec\nu(y)+\xi)d\mathcal{H}^{N-1}(\xi)dt
\bigg)d\tau dsdw\Bigg)\big|u^+(y)-u^-(y)\big|^q
d\mathcal{H}^{N-1}(y)\,.
\end{multline}
Introducing the notation
\begin{equation}
  \label{eq:1}
  \Lambda(y,a,b)=\int_a^b \int_{H_0(\vec\nu(y))}\eta(t\vec\nu(y)+\xi)\,d\mathcal{H}^{N-1}(\xi)\,dt
\end{equation}
allows us to rewrite
\eqref{fgyufghfghjgghgjkhkkGHGHKKokuhhhugugzzkhhbvqkkklkljjkjkklkljkll;k;k;ouuiojkkjhkhkljl;klk;k;jhkhkhkhkjljjk;kk;jkkhihk}
as
\begin{multline}\label{fjhgjghkhkljhljljljlj}
\lim_{\e\to 0^+}\frac{1}{-\ln{\e}}\|u_\e\|^q_{W^{1/q,q}}=\\
\int_{J_u\cap\Omega}\Bigg\{\int_{\R^{N-1}}\int_{\R}\int_{\R}\frac{1}{\tau^2}\frac{1}{|\tau|^{N-1}}\frac{1}{\big(\sqrt{1+|w/|\tau||^2}\big)^{N+1}}\times\\
\Bigg(\int_{H_0(\vec\nu(y))}\bigg( \eta\big(s\vec\nu(y)+\zeta\big)s
-\eta\Big(\big(s-\tau\big)\vec\nu(y)+\zeta\Big)\big(s-\tau\big)\bigg)d\mathcal{H}^{N-1}(\zeta)\Bigg) \times\\
\times \frac{dG_q}{d\rho}\Big(\Lambda(y,s-\tau,s)\Big)d\tau
dsdw\Bigg\}\big|u^+(y)-u^-(y)\big|^q d\mathcal{H}^{N-1}(y).
\end{multline}
The change of variables $w/|\tau|\to v$ in the R.H.S.~of
\er{fjhgjghkhkljhljljljlj} gives
\begin{multline}\label{fgyufghfghjgghgjkhkkGHGHKKokuhhhugugzzkhhbvqkkklkljjkjkklkljkll;k;k;ouuiojkkjhkhkljl;klk;k;jhkhkhkhkjljjk;kk;jkkhihkjljllmnjghghg}
\lim_{\e\to 0^+}\frac{1}{-\ln{\e}}\|u_\e\|^q_{W^{1/q,q}} =\\
D_N\int_{J_u\cap\Omega}\Bigg(\int_{\R}\int_{\R}\frac{1}{\tau^2}\Bigg(
\int_{H_0(\vec\nu(y))}\bigg( \eta\big(s\vec\nu(y)+\zeta\big)s
-\eta\Big(\big(s-\tau\big)\vec\nu(y)+\zeta\Big)\big(s-\tau\big)\bigg)d\mathcal{H}^{N-1}(\zeta)\Bigg) \times\\
\times \frac{dG_q}{d\rho}\Big(\Lambda(y,s-\tau,s)\Big) d\tau
ds\Bigg)\big|u^+(y)-u^-(y)\big|^q d\mathcal{H}^{N-1}(y),
\end{multline}
where $D_N$ is the dimensional constant given by
\begin{equation}\label{fgyufghfghjgghgjkhkkGHGHKKokuhhhugugzzkhhbvqkkklkljjkjkklkljkll;k;k;ouuiojkkjhkhkljl;klk;k;jhkhkhkhkjljjk;kk;jkkhihkjljllmnjghghgjkhjh}
D_N:=\int_{\R^{N-1}}\frac{1}{\big(\sqrt{1+|v|^2}\big)^{N+1}}dv.
\end{equation}
%
%
%
%
%
%
Then we rewrite
\er{fgyufghfghjgghgjkhkkGHGHKKokuhhhugugzzkhhbvqkkklkljjkjkklkljkll;k;k;ouuiojkkjhkhkljl;klk;k;jhkhkhkhkjljjk;kk;jkkhihkjljllmnjghghg}
as
\begin{multline}\label{fgyufghfghjgghgjkhkkGHGHKKokuhhhugugzzkhhbvqkkklkljjkjkklkljkll;k;k;ouuiojkkjhkhkljl;klk;k;jhkhkhkhkjljjk;kk;jkkhihkjljllmnjghghgjhjhjh}
\lim_{\e\to 0^+}\frac{1}{-\ln{\e}}\|u_\e\|^q_{W^{1/q,q}} =\\
\lim_{M\to
+\infty}\Bigg(D_N\int_{J_u\cap\Omega}\Bigg(\int_{\R}\int_{-M}^{M}\frac{1}{\tau^2}\bigg(
\int_{H_0(\vec\nu(y))}
s\Big(\eta\big(s\vec\nu(y)+\zeta\big)-\eta\big((s-\tau)\vec\nu(y)+\zeta\big)\Big)d\mathcal{H}^{N-1}(\zeta)\bigg) \times\\
\times
\frac{dG_q}{d\rho} \Big(\Lambda(y,s-\tau,s)\Big) d\tau ds\Bigg)\big|u^+(y)-u^-(y)\big|^q d\mathcal{H}^{N-1}(y)\\
+D_N\int_{J_u\cap\Omega}\Bigg(\int_{\R}\int_{-M}^{M}\frac{1}{\tau}\bigg(
\int_{H_0(\vec\nu(y))}
\eta\big((s-\tau)\vec\nu(y)+\zeta\big)d\mathcal{H}^{N-1}(\zeta)\bigg) \times\\
\times \frac{dG_q}{d\rho} \Big(\Lambda(y,s-\tau,s)\Big) d\tau
ds\Bigg)\big|u^+(y)-u^-(y)\big|^q d\mathcal{H}^{N-1}(y)\Bigg).
\end{multline}
Integration by parts of
\er{fgyufghfghjgghgjkhkkGHGHKKokuhhhugugzzkhhbvqkkklkljjkjkklkljkll;k;k;ouuiojkkjhkhkljl;klk;k;jhkhkhkhkjljjk;kk;jkkhihkjljllmnjghghgjhjhjh}
and using \er{fhjgjgfgjjkjhhjhjhhjh} give
\begin{multline}\label{11aafgyufghfghjgghgjkhkkGHGHKKokuhhhugugzzkhhbvqkkklkljjkjkklkljkll;k;k;ouuiojkkjhkhkljl;klk;k;jhkhkhkhkjljjk;kk;jkkhihkjljllmjhnjghghgjhjhjhhjhj}
\lim_{\e\to 0^+}\frac{1}{-\ln{\e}}\|u_\e\|^q_{W^{1/q,q}} =\\
-\lim_{M\to +\infty}D_N\int_{J_u\cap \Omega}\big|u^+(y)-u^-(y)\big|^q\Bigg(
\int_{\R}\int_{-M}^{M}\frac{1}{\tau^2}
\Big|\Lambda(y,s-\tau,s)
\Big|^qd\tau ds\Bigg) d\mathcal{H}^{N-1}(y)\\
+ \lim_{M\to +\infty}D_N\int_{J_u\cap\Omega}\Bigg(
\int_{\R}\int_{-M}^{M}\frac{1}{\tau^2} \Big|\Lambda(y,s-\tau,s)
\Big|^qd\tau ds\Bigg)\big|u^+(y)-u^-(y)\big|^q
d\mathcal{H}^{N-1}(y)\\ + \lim_{M\to
+\infty}\frac{D_N}{M}\int_{J_u\cap\Omega}\Bigg(
\int_{\R}\Big|\Lambda(y,s-M,s)\Big|^qds
+\int_{\R}\Big|\Lambda(y,s,s+M)\Big|^q
ds\Bigg)\big|u^+(y)-u^-(y)\big|^q d\mathcal{H}^{N-1}(y)\\
=\lim_{M\to
+\infty}\frac{D_N}{M}\int_{J_u\cap\Omega}\Bigg(
\int_{\R}\Big|\Lambda(y,s-M,s)\Big|^qds
+\int_{\R}\Big|\Lambda(y,s,s+M)\Big|^q
ds\Bigg)\big|u^+(y)-u^-(y)\big|^q d\mathcal{H}^{N-1}(y).
\end{multline}
Therefore, applying  L'H\^{o}pital's rule in
\er{11aafgyufghfghjgghgjkhkkGHGHKKokuhhhugugzzkhhbvqkkklkljjkjkklkljkll;k;k;ouuiojkkjhkhkljl;klk;k;jhkhkhkhkjljjk;kk;jkkhihkjljllmjhnjghghgjhjhjhhjhj},
using
\er{fhjgjgfgjjkjhhjhjhhjh}, we
deduce that
\begin{multline}\label{11ddfgyufghfghjgghgjkhkkGHGHKKokuhhjujkhhhugugzzkhhbvqkkklkljjkjkjhjgklkljkll;k;k;ouuiojkkjhkhkljl;klk;k;jhkhkhjjkhkjljjk;kk;jkkhihkjljllmjhnjghghgjhjhjhhjhj}
\lim_{\e\to 0^+}\frac{1}{-\ln{\e}}\|u_\e\|^q_{W^{1/q,q}} =\\
\lim_{M\to
+\infty}D_N\int_{J_u\cap\Omega}\Bigg(
\int_{\R}
\frac{dG_q}{d\rho}\Big(\Lambda(y,s-M,s)\Big)
\bigg(\int_{H_0(\vec\nu(y))}\eta\big((s-M)\vec\nu(y)+\xi\big)d\mathcal{H}^{N-1}(\xi)\bigg)ds
\\+\int_{\R}\frac{dG_q}{d\rho}\Big(\Lambda(y,s,s+M)\Big)\bigg(\int_{H_0(\vec\nu(y))}\eta\big((s+M)\vec\nu(y)+\xi\big)d\mathcal{H}^{N-1}(\xi)\bigg)\Bigg)
ds\\ \times\big|u^+(y)-u^-(y)\big|^q d\mathcal{H}^{N-1}(y).
\end{multline}
Changing variables of integration we rewrite
\er{11ddfgyufghfghjgghgjkhkkGHGHKKokuhhjujkhhhugugzzkhhbvqkkklkljjkjkjhjgklkljkll;k;k;ouuiojkkjhkhkljl;klk;k;jhkhkhjjkhkjljjk;kk;jkkhihkjljllmjhnjghghgjhjhjhhjhj}
as
\begin{multline}\label{2211ddfgyufghfghjgghgjkhkkGHGHKKokuhhjujkhhhugugzzkhhbvqkkklkljjkjkjhjgklkljkll;k;k;ouuiojkkjhkhkljl;klk;k;jhkhkhjjkhkjljjk;kk;jkkhihkjljllmjhnjghghgjhjhjhhjhj}
\lim_{\e\to 0^+}\frac{1}{-\ln{\e}}\|u_\e\|^q_{W^{1/q,q}} =\\
\lim_{M\to +\infty}D_N\int_{J_u\cap\Omega}\Bigg(
\int_{\R}\frac{dG_q}{d\rho}\Big(\Lambda(y,s,s+M)\Big)
\bigg(\int_{H_0(\vec\nu(y))}\eta\big(s\vec\nu(y)+\xi\big)d\mathcal{H}^{N-1}(\xi)\bigg)ds\\
+\int_{\R}\frac{dG_q}{d\rho}\Big(\Lambda(y,s-M,s)\Big)\bigg(\int_{H_0(\vec\nu(y))}\eta\big(s\vec\nu(y)+\xi\big)d\mathcal{H}^{N-1}(\xi)\bigg)
ds\Bigg)\\
\times\big|u^+(y)-u^-(y)\big|^q d\mathcal{H}^{N-1}(y)\\
=
D_N\int_{J_u\cap\Omega}\Bigg(
\int_{\R}\frac{dG_q}{d\rho}\Big(\Lambda(y,s,\infty)\Big)
\bigg(\int_{H_0(\vec\nu(y))}\eta\big(s\vec\nu(y)+\xi\big)d\mathcal{H}^{N-1}(\xi)\bigg)ds\\
+\int_{\R}\frac{dG_q}{d\rho}\Big(\Lambda(y,-\infty,s)\Big)\bigg(\int_{H_0(\vec\nu(y))}\eta\big(s\vec\nu(y)+\xi\big)d\mathcal{H}^{N-1}(\xi)\bigg)
ds\Bigg)\\
\times\big|u^+(y)-u^-(y)\big|^q d\mathcal{H}^{N-1}(y).
\end{multline}
 Applying Newton-Leibniz formula in
\er{2211ddfgyufghfghjgghgjkhkkGHGHKKokuhhjujkhhhugugzzkhhbvqkkklkljjkjkjhjgklkljkll;k;k;ouuiojkkjhkhkljl;klk;k;jhkhkhjjkhkjljjk;kk;jkkhihkjljllmjhnjghghgjhjhjhhjhj}
and using \er{fhjgjgfgjjkjhhjhjhhjh} we obtain that
\begin{multline}\label{fgyufghfghjgghgjkhkkGHGHKKokuhhjujkjhjhhhhugugzzkhhbvqkkklkljjkjkjhjgklkljkll;k;k;ouuiojkkjhkhkljl;klk;k;jhkhkhjjkhkjljjk;kk;jkkhihkjljllmjhnjghghgjhjhjhhjhj}
\lim_{\e\to
0^+}\frac{1}{-\ln{\e}}\|u_\e\|^q_{W^{1/q,q}}=\\2D_N\int_{J_u\cap\Omega}
\bigg|\int_{-\infty}^{\infty}\int_{H_0(\vec\nu(y))}\eta(t\vec\nu(y)+\xi)d\mathcal{H}^{N-1}(\xi)dt\bigg|^q
\big|u^+(y)-u^-(y)\big|^q d\mathcal{H}^{N-1}(y)\\=
2D_N\bigg|\int_{\R^N}\eta(z)dz\bigg|^q\int_{J_u\cap\Omega}
\big|u^+(y)-u^-(y)\big|^q d\mathcal{H}^{N-1}(y)\,,
\end{multline}
and
\er{fgyufghfghjgghgjkhkkGHGHKKokuhhjujkjhjhhhhugugzzkhhbvqkkklkljjkjkjhjgklkljkll;k;k;ouuiojkkjhkhkljl;klk;k;jhkhkhjjkhkjljjk;kk;jkkhihkjljllmjhnjghghgjhjhjhhjhjRR} follows.
\end{proof}
\begin{corollary}\label{jbnvjnnjvnvnnbRR}
Let $q>1$ and let $\Omega\subset\R^N$ be an open set. Assume
$W:\R^d\times\R^N\to\R$ is a Borel measurable function such
that, $W(0,\cdot)\in L^1(\Omega,\R)$ and for every $D>0$ there exists
$C:=C_D>0$ such that
\begin{equation}\label{jmvnvnbccbvhjhjhhjjkhgjgGHKKzzbvqkhhihkRRruuyrfhjkkhhhhhj}
\big|W(b,x)-W(a,x)\big|\leq C_D|b-a|\quad\quad\forall
x\in\R^N,\;\forall\, a,b\in B_D(0).
\end{equation}
Let $u\in BV(\R^N,\R^d)\cap L^\infty(\R^N,\R^d)$ be
such that $\|Du\|(\partial\Omega)=0$ and $W\big(u(x),x\big)=0$ a.e.
in $\Omega$.
Let $\eta\in C^\infty_c(\R^N,\R)$ be such that
$\int_{\R^N}\eta(z)dz=1$ and $\supp\eta\subset B_R(0)$. For every
$\rho>0$ set
\begin{equation}\label{jmvnvnbccbvhjhjhhjjkhgjgGHKKzzbvqkhhihkRRruuyrfh}
\eta_{\rho}(z):=\frac{1}{\rho^N}\eta\Big(\frac{z}{\rho}\Big)\quad\forall
z\in\R^N.
\end{equation}
Finally, for every $x\in\R^N$ and every $\e>0$ define
\begin{equation}\label{jmvnvnbccbvhjhjhhjjkhgjgGHKKzzbvqkhhihkRR}
u_{\rho,\e}(x):=\frac{1}{\e^N}\int_{\R^N}\eta_\rho\Big(\frac{y-x}{\e}\Big)u(y)dy=\int_{\R^N}\eta(z)u(x+\e\rho
z)dz=\int_{B_R(0)}\eta(z)u(x+\e \rho z)dz.
\end{equation}
Then,
\begin{multline}\label{fgyufghfghjgghgjkhkkGHGHKKokuhhjujkjhjhhhhugugzzkhhbvqkkklkljjkjkjhjgklkljkll;k;k;ouuiojkkjhkhkljl;klk;k;jhkhkhjjkhkjljjk;kk;jkkhihkjljllmjhnjghghgjhjhjhhjhjRRRR}
\lim_{\rho\to 0^+}\Bigg\{\limsup_{\e\to
0^+}\bigg(\frac{1}{-\ln{(\e)}}\Big(\|u_{\rho,\e}\|^q_{W^{1/q,q}(\R^N,\R^d)}-\|u_{\rho,\e}\|^q_{W^{1/q,q}(\R^N\setminus\ov\Omega,\R^d)}\Big)
+\frac{1}{\e}\int_{\Omega}W\Big(u_{\rho,\e}(x),x\Big)dx
\bigg)\Bigg\}\\=\lim_{\rho\to 0^+}\Bigg\{\limsup_{\e\to
0^+}\bigg(\frac{1}{-\ln{(\e)}}\|u_{\rho,\e}\|^q_{W^{1/q,q}(\Omega,\R^d)}+\frac{1}{\e}\int_{\Omega}W\Big(u_{\rho,\e}(x),x\Big)dx
\bigg)\Bigg\}
\\=
\Bigg(\int_{\R^{N-1}}\frac{2}{\big(\sqrt{1+|v|^2}\big)^{N+1}}dv\Bigg)\int_{J_u\cap\Omega}
\big|u^+(y)-u^-(y)\big|^q d\mathcal{H}^{N-1}(y).
\end{multline}
\end{corollary}
\begin{proof}
Since $\int_{\R^N}\eta_\rho(z)dz=1$, applying Proposition
\ref{jbnvjnnjvnvnnb}, first for $\R^N$, then for
$\R^N\setminus\ov\Omega$, and finally for $\Omega$, yields, for every $\rho>0$,
\begin{multline}\label{fgyufghfghjgghgjkhkkGHGHKKokuhhjujkjhjhhhhugugzzkhhbvqkkklkljjkjkjhjgklkljkll;k;k;ouuiojkkjhkhkljl;klk;k;jkjjhkhkhjjkhkjljjk;kk;jkkhihkjljllmjhnjghghgjhjhjhhjhjRRRRhjjgj}
\lim_{\e\to0^+}
\frac{1}{-\ln{(\e)}}\bigg(\|u_{\rho,\e}\|^q_{W^{1/q,q}(\R^N,\R^d)}-\|u_{\rho,\e}\|^q_{W^{1/q,q}(\R^N\setminus\ov\Omega,\R^d)}
\bigg)
\\=2D_N\left(\int_{J_u}
\big|u^+(y)-u^-(y)\big|^q d\mathcal{H}^{N-1}(y)-\int_{J_u\cap(\R^N\setminus\ov\Omega)}
\big|u^+(y)-u^-(y)\big|^q d\mathcal{H}^{N-1}(y)\right)\\
=2D_N \int_{J_u\cap\Omega} \big|u^+(y)-u^-(y)\big|^q
d\mathcal{H}^{N-1}(y)= \lim_{\e\to
0^+}\bigg(\frac{1}{-\ln{(\e)}}\|u_{\rho,\e}\|^q_{W^{1/q,q}(\Omega,\R^d)}
\bigg),
\end{multline}
 where $D_N$ is the constant defined in \er{fgyufghfghjgghgjkhkkGHGHKKokuhhhugugzzkhhbvqkkklkljjkjkklkljkll;k;k;ouuiojkkjhkhkljl;klk;k;jhkhkhkhkjljjk;kk;jkkhihkjljllmnjghghgjkhjh}.
On the other hand, since $W\big(u(x),x\big)=0$ a.e. in $\Omega$ and
$u\in L^\infty$, by
\er{jmvnvnbccbvhjhjhhjjkhgjgGHKKzzbvqkhhihkRRruuyrfhjkkhhhhhj} we
 get that
\begin{multline}\label{ffffgyufghfghjgghgjkhkkGHGHKKokuhhjujkjhjhhhhugugzzkhhbvqkkklkljjkjkjhjgklkljkll;k;k;ouuiojkkjhkhkljl;klk;k;jhkhkhjjkhkjljjk;kk;jkkhihkjljllmjhnjghghgjhjhjhhjhjRRRRjkhjh}
\bigg|\frac{1}{\e}\int_{\Omega}W\Big(u_{\rho,\e}(x),x\Big)dx
\bigg|=\bigg|\frac{1}{\e}\int_{\Omega}\Big(W\big(u_{\rho,\e}(x),x\big)-W\big(u(x),x\big)\Big)dx
\bigg| \leq
C\int_{\R^N}\frac{1}{\e}\Big|u_{\rho,\e}(x)-u(x)\Big|dx\\ \leq
C\int_{B_R(0)}\big|\eta(z)\big|\Bigg(\int_{\R^N}\frac{1}{\e}\Big|u(x+\e\rho
z)-u(x)\Big|dx\Bigg)dz
\\=
C\rho\int_{B_R(0)}|z|\big|\eta(z)\big|\Bigg(\int_{\R^N}\frac{1}{\e\rho|z|}\Big|u(x+\e\rho
z)-u(x)\Big|dx\Bigg)dz,
\end{multline}
for some constant $C>0$, independent of $\e$ and $\rho$. Thus, taking
into account the following well known uniform bound from the theory
of $BV$ functions,
\begin{equation}\label{jiouiuuiui}
\int_{\R^N}\frac{1}{\rho\e|z|}\Big|u(x+\rho\e z)-u(x)\Big|dx\leq
C_0\|Du\|(\R^N)\quad\forall
z\in\R^N,\,\forall\rho,\,\e>0,
\end{equation} we
obtain that
\begin{equation}\label{fgyufghfghjgghgjkhkkGHGHKKokuhhjujkjhjhhhhugugzzkhhbvqkkklkljjkjkjhjgklkljkll;k;k;ouuiojkkjhkhkljl;klk;k;jhkhkhjjkhkjljjk;kk;jkkhihkjljllmjhnjghghgjhjhjhhjhjRRRRjkhjh}
\limsup_{\e\to
0^+}\bigg|\frac{1}{\e}\int_{\Omega}W\Big(u_{\rho,\e}(x),x\Big)dx
\bigg|\leq
CC_0\|Du\|(\R^N)\rho\int_{B_R(0)}|z|\big|\eta(z)\big|dz=O(\rho).
\end{equation}
By
\er{fgyufghfghjgghgjkhkkGHGHKKokuhhjujkjhjhhhhugugzzkhhbvqkkklkljjkjkjhjgklkljkll;k;k;ouuiojkkjhkhkljl;klk;k;jhkhkhjjkhkjljjk;kk;jkkhihkjljllmjhnjghghgjhjhjhhjhjRRRRjkhjh}
and
\er{fgyufghfghjgghgjkhkkGHGHKKokuhhjujkjhjhhhhugugzzkhhbvqkkklkljjkjkjhjgklkljkll;k;k;ouuiojkkjhkhkljl;klk;k;jkjjhkhkhjjkhkjljjk;kk;jkkhihkjljllmjhnjghghgjhjhjhhjhjRRRRhjjgj}
we finally derive
\er{fgyufghfghjgghgjkhkkGHGHKKokuhhjujkjhjhhhhugugzzkhhbvqkkklkljjkjkjhjgklkljkll;k;k;ouuiojkkjhkhkljl;klk;k;jhkhkhjjkhkjljjk;kk;jkkhihkjljllmjhnjghghgjhjhjhhjhjRRRR}.
\end{proof}

\begin{proof}[Proof of Theorem \ref{jbnvjnnjvnvnnbRRkkjj}]
Let $\eta,\eta_\rho$ and $u_{\rho,\e}$ be defined as in Corollary
\ref{jbnvjnnjvnvnnbRR}. Then $u_{\rho,\e}\in C^\infty(\R^N,\R^d)\cap
W^{1,1}(\R^N,\R^d)\cap W^{1,\infty}(\R^N,\R^d)$ and by Corollary
\ref{jbnvjnnjvnvnnbRR} we have
\begin{multline}\label{fgyufghfghjgghgjkhkkGHGHKKokuhhjujkjhjhhhhugugzzkhhbvqkkklkljjkjkjhjgklkljkll;k;k;ouuiojkkjhkhkljl;klk;k;jhkhkhjjkhkjljjk;kk;jkkhihkjljllmjhnjghghgjhjhjhhjhjRRRRkkijhjh}
\lim_{\rho\to 0^+}\Bigg\{\limsup_{\e\to
0^+}\bigg(\frac{1}{-\ln{(\e)}}\Big(\|u_{\rho,\e}\|^q_{W^{1/q,q}(\R^N,\R^d)}-\|u_{\rho,\e}\|^q_{W^{1/q,q}(\R^N\setminus\ov\Omega,\R^d)}\Big)
+\frac{1}{\e}\int_{\Omega}W\Big(u_{\rho,\e}(x),x\Big)dx
\bigg)\Bigg\}\\=\lim_{\rho\to 0^+}\Bigg\{\limsup_{\e\to
0^+}\bigg(\frac{1}{-\ln{\e}}\|u_{\rho,\e}\|^q_{W^{1/q,q}(\Omega,\R^d)}+\frac{1}{\e}\int_{\Omega}W\Big(u_{\rho,\e}(x),x\Big)dx
\bigg)\Bigg\}
\\=
\Bigg(\int_{\R^{N-1}}\frac{2}{\big(\sqrt{1+|v|^2}\big)^{N+1}}dv\Bigg)\int_{J_u\cap\Omega}
\big|u^+(y)-u^-(y)\big|^q d\mathcal{H}^{N-1}(y).
\end{multline}
Clearly, for every $x\in\R^N$ we have $A\cdot\nabla
u_{\rho,\e}(x)=0$ and $u_{\rho,\e}(x)\to u(x)$ strongly in
$L^p(\R^N,\R^d)$ as $\e\to 0^+$ for every fixed $\rho$ and $p$.
Therefore, by the above and by
\er{fgyufghfghjgghgjkhkkGHGHKKokuhhjujkjhjhhhhugugzzkhhbvqkkklkljjkjkjhjgklkljkll;k;k;ouuiojkkjhkhkljl;klk;k;jhkhkhjjkhkjljjk;kk;jkkhihkjljllmjhnjghghgjhjhjhhjhjRRRRkkijhjh}
we can complete the proof of the first assertion of the theorem using a standard
diagonal argument.

It remains to show the second assertion of the theorem, namely, that
in the case $A\equiv 0$  we can construct $\psi_\e$ satisfying the
additional condition \er{jmvnvnbccbvhjhjhhjjkhgjgGHKKzzbvqkhhihkRRkkhihhjhjhjhkk}. Let $\varphi\in C_c^\infty(\R^N,\R)$ be such that
$\int_\Omega\varphi(x)dx=1$. Define
\begin{equation}\label{jmvnvnbccbvhjhjhhjjkhgjgGHKKzzbvqkhhihkRRkkhihh}
\tilde u_{\rho,\e}(x):=u_{\rho,\e}(x)-\varphi(x)c_{\e,\rho},
\end{equation}
where
\begin{equation}\label{jmvnvnbccbvhjhjhhjjkhgjgGHKKzzbvqkhhihkRRkkhihhjjj}
c_{\e,\rho}:=\int_\Omega u_{\rho,\e}(y)dy-\int_\Omega u(y)dy.
\end{equation}
In particular,
\begin{equation}\label{jmvnvnbccbvhjhjhhjjkhgjgGHKKzzbvqkhhihkRRkkhihhjhjhjhjkhhj}
\int_\Omega \tilde u_{\rho,\e}(x)dx=\int_\Omega u(x)dx,
\end{equation}
and $\lim_{\e\to 0^+}c_{\e,\rho}=0$. On the other hand, since
$W\big(u(x),x\big)=0$ a.e. in $\Omega$, $W(b,x)$ is nonnegative and
$W(b,x)$ is differentiable with respect to the $b$ variable, we have
\begin{equation}\label{jmvnvnbccbvhjhjhhjjkhgjgGHKKzzbvqkhhihkRRkkhihhjhjhjhjkhhjhhyfyfy}
\nabla_b W\big(u(x),x\big)=0\quad\text{a.e. in }\Omega.
\end{equation}
Thus, since $u\in L^\infty$, by
\er{jmvnvnbccbvhjhjhhjjkhgjgGHKKzzbvqkhhihkRRkkhihh} we get that
\begin{multline}\label{ffffgyufghfghjgghgjkhkkGHGHKKokuhhjujkjhjhhhhugugzzkhhbvqkjijjikkklkljjkjkjhjgklkljkll;k;k;ouuiojkkjhkhkljl;klk;k;jhkhkhjjkhkjljjk;kk;jkkhihkjljllmjhnjghghgjhjhjhhjhjRRRRjkhjh}
\bigg|\frac{1}{\e}\int\limits_{\Omega}\Big(W\big(\tilde
u_{\rho,\e}(x),x\big)-W\big(u_{\rho,\e}(x),x\big)\Big)dx \bigg|=
\Bigg|\frac{c_{\e,\rho}}{\e}\cdot\int\limits_0^1\int\limits_{\Omega}\nabla_b
W\Big(u_{\rho,\e}(x)-s\varphi(x)c_{\e,\rho},x\Big)\varphi(x)dx
ds\Bigg|
\\
\leq
C\Bigg(\int_{\R^N}\frac{1}{\e}\Big|u_{\rho,\e}(x)-u(x)\Big|dx\Bigg)\Bigg|\int_0^1\int_{\Omega}\nabla_b
W\Big(u_{\rho,\e}(x)-s\varphi(x)c_{\e,\rho},x\Big)\varphi(x)dx
ds\Bigg|\\
\leq
C\Bigg(\int_{B_R(0)}\big|\eta(z)\big|\bigg(\int_{\R^N}\frac{1}{\e}\Big|u(x+\e\rho
z)-u(x)\Big|dx\bigg)dz\Bigg)\times\\
\times\Bigg|\int_0^1\int_{\Omega}\nabla_b
W\Big(u_{\rho,\e}(x)-s\varphi(x)c_{\e,\rho},x\Big)\varphi(x)dx
ds\Bigg|
\\=
C\rho\Bigg(\int_{B_R(0)}|z|\big|\eta(z)\big|\bigg(\int_{\R^N}\frac{1}{\e\rho|z|}\Big|u(x+\e\rho
z)-u(x)\Big|dx\bigg)dz\Bigg)\times\\
\times\Bigg|\int_0^1\int_{\Omega}\nabla_b
W\Big(u_{\rho,\e}(x)-s\varphi(x)c_{\e,\rho},x\Big)\varphi(x)dx
ds\Bigg|.
\end{multline}
On the other hand, taking into account \er{jiouiuuiui} and using the
Dominated Convergence Theorem and
\er{jmvnvnbccbvhjhjhhjjkhgjgGHKKzzbvqkhhihkRRkkhihhjhjhjhjkhhjhhyfyfy},
we obtain that
\begin{multline}\label{ffffgyufghfghjgghgjkhkkGHGHKKokuhhjujkjhjhhhhugugzzkhhbvqkkklkljjkjkjhjgklkljkll;k;k;ouuiojkkjhkhkljl;klk;k;jhkhkhjjkhkjljjk;kk;jkkhihkjljllmjhnjghghgjhjhjhhjhjRRRRjkhjhojjiji}
\limsup_{\e\to
0+}\Bigg(\int_{B_R(0)}|z|\big|\eta(z)\big|\bigg(\int_{\R^N}\frac{1}{\e\rho|z|}\Big|u(x+\e\rho
z)-u(x)\Big|dx\bigg)dz\Bigg)\times\\
\times\Bigg|\int_0^1\int_{\Omega}\nabla_b
W\Big(u_{\rho,\e}(x)-s\varphi(x)c_{\e,\rho},x\Big)\varphi(x)dx
ds\Bigg|\leq C_0\Big(\|Du\|(\R^n)\Big)\Bigg(\int_{B_R(0)}|z|\big|\eta(z)\big|dz\Bigg)\times\\
\times\Bigg|\int_0^1\int_{\Omega}\nabla_b W\Big(\lim_{\e\to
0+}u_{\rho,\e}(x)-s\varphi(x)\lim_{\e\to
0+}c_{\e,\rho}\,,\,x\Big)\varphi(x)dx ds\Bigg|\\
=
C_0\Big(\|Du\|(\R^n)\Big)\Bigg(\int_{B_R(0)}|z|\big|\eta(z)\big|dz\Bigg)\Bigg|\int_{\Omega}\nabla_b
W\Big(u(x),x\Big)\varphi(x)dx\Bigg|=0.
\end{multline}
Using
\er{ffffgyufghfghjgghgjkhkkGHGHKKokuhhjujkjhjhhhhugugzzkhhbvqkkklkljjkjkjhjgklkljkll;k;k;ouuiojkkjhkhkljl;klk;k;jhkhkhjjkhkjljjk;kk;jkkhihkjljllmjhnjghghgjhjhjhhjhjRRRRjkhjhojjiji}
in
\er{ffffgyufghfghjgghgjkhkkGHGHKKokuhhjujkjhjhhhhugugzzkhhbvqkjijjikkklkljjkjkjhjgklkljkll;k;k;ouuiojkkjhkhkljl;klk;k;jhkhkhjjkhkjljjk;kk;jkkhihkjljllmjhnjghghgjhjhjhhjhjRRRRjkhjh}
yields
\begin{equation}\label{ffffgyufghfghjgghgjkhkkGHGHKKokuhhjujkjhjhhhhugugzzkhhbvqkjijjikkklkljjkjkjhjgklkljkll;k;k;ouuiojkkjhkhkljl;klk;k;jhkhkhjjkhkjljjk;kk;jkkhihkjljllmjhnjghghgjhjhjhhjhjRRRRjkhjhhjjg}
\limsup_{\e\to 0+}\bigg|\frac{1}{\e}\int_{\Omega}\Big(W\big(\tilde
u_{\rho,\e}(x),x\big)-W\big(u_{\rho,\e}(x),x\big)\Big)dx \bigg|=0.
\end{equation}
Plugging
\er{ffffgyufghfghjgghgjkhkkGHGHKKokuhhjujkjhjhhhhugugzzkhhbvqkjijjikkklkljjkjkjhjgklkljkll;k;k;ouuiojkkjhkhkljl;klk;k;jhkhkhjjkhkjljjk;kk;jkkhihkjljllmjhnjghghgjhjhjhhjhjRRRRjkhjhhjjg}
into
\er{fgyufghfghjgghgjkhkkGHGHKKokuhhjujkjhjhhhhugugzzkhhbvqkkklkljjkjkjhjgklkljkll;k;k;ouuiojkkjhkhkljl;klk;k;jhkhkhjjkhkjljjk;kk;jkkhihkjljllmjhnjghghgjhjhjhhjhjRRRRkkijhjh}
we get that
\begin{multline}\label{fgyufghfghjgghgjkhkkGHGHKKokuhhjujkjhjhhhhugugzzkhhbvqkkklkhjggjgljjkjkjhjgklkljkll;k;k;ouuiojkkjhkhkljl;klk;k;jhkhkhjjkhkjljjk;kk;jkkhihkjljllmjhnjghghgjhjhjhhjhjRRRRkkijhjh}
\lim_{\rho\to 0^+}\Bigg\{\limsup_{\e\to
0^+}\bigg(\frac{1}{-\ln{(\e)}}\Big(\|\tilde
u_{\rho,\e}\|^q_{W^{1/q,q}(\R^N,\R^d)}-\|\tilde
u_{\rho,\e}\|^q_{W^{1/q,q}(\R^N\setminus\ov\Omega,\R^d)}\Big)
+\frac{1}{\e}\int_{\Omega}W\Big(\tilde u_{\rho,\e}(x),x\Big)dx
\bigg)\Bigg\}\\=\lim_{\rho\to 0^+}\Bigg\{\limsup_{\e\to
0^+}\bigg(\frac{1}{-\ln{\e}}\|\tilde
u_{\rho,\e}\|^q_{W^{1/q,q}(\Omega,\R^d)}+\frac{1}{\e}\int_{\Omega}W\Big(\tilde
u_{\rho,\e}(x),x\Big)dx \bigg)\Bigg\}
\\=
\Bigg(\int_{\R^{N-1}}\frac{2}{\big(\sqrt{1+|v|^2}\big)^{N+1}}dv\Bigg)\int_{J_u\cap\Omega}
\big|u^+(y)-u^-(y)\big|^q d\mathcal{H}^{N-1}(y).
\end{multline}
Moreover, $\tilde u_{\rho,\e}\to u$ strongly in $L^p(\R^N,\R^d)$ as
$\e\to 0^+$ for every fixed $\rho$ and $p$. Therefore, by the above
and
\er{fgyufghfghjgghgjkhkkGHGHKKokuhhjujkjhjhhhhugugzzkhhbvqkkklkhjggjgljjkjkjhjgklkljkll;k;k;ouuiojkkjhkhkljl;klk;k;jhkhkhjjkhkjljjk;kk;jkkhihkjljllmjhnjghghgjhjhjhhjhjRRRRkkijhjh}
we complete again the proof by a standard diagonal argument.
\end{proof}

The next lemma is needed for the proof of Theorem~\ref{jbnvjnnjvnvnnbhh}
(in the general case $\eta\in W^{1,1}$).

\begin{lemma}\label{hjgjhfhfhfh}
Let $\Omega\subset\R^N$ be an open set and let $u\in BV(\R^N,\R^d)\cap
L^\infty(\R^N,\R^d)$. For $\eta\in
W^{1,1}(\R^N,\R)$,
every $x\in\R^N$ and every $\e>0$ define
\begin{equation}\label{jmvnvnbccbvhjhjhhjjkhgjgGHKKzzbvqkhhihkjjjkktt}
u_\e(x):=\frac{1}{\e^N}\int_{\R^N}\eta\Big(\frac{y-x}{\e}\Big)u(y)dy=\int_{\R^N}\eta(z)u(x+\e
z)dz.
\end{equation}
Then, for every $q>1$ and for every $\e\in(0,1)$ we have
\begin{multline}\label{fgyufghfghjgghgjkhkkGHGHKKokuhhjujkhghghgjhjhhhhugugzzkhhbvqkkklkljjkjkjhjgklkljkll;k;k;ouuiojkkjhkhkljl;klk;k;jhkhkhjjkhkjljjk;kk;jkkhihkjljllmjhnjghghgjhjhjhhjhjRRjjjhhkkhjhttuhgugjojjjggugggujkkhjhghghg}
\frac{1}{\omega_{N-1}\big|\ln{\e}\big|}\int_{\Omega}\bigg(\int_{\Omega}\frac{\big|u_\e(x)-u_\e(y)\big|^q}{|x-y|^{N+1}}dy\bigg)dx
\leq
\frac{2^q\|u\|_{L^1(\R^N,\R^d)}\|u\|^{q-1}_{L^\infty(\R^N,\R^d)}\|\eta\|^q_{L^1(\R^N,\R)}}{\big|\ln{\e}\big|}\\+
\frac{\big(3\|u\|_{L^\infty(\R^N,\R^d)}\|\eta\|_{W^{1,1}(\R^N,\R)}\big)^{q-1}\|\eta\|_{L^{1}(\R^N,\R)}\|Du\|(\R^N)}{(q-1)\big|\ln{\e}\big|}
\\+
\big(3\|u\|_{L^\infty(\R^N,\R^d)}\|\eta\|_{W^{1,1}(\R^N,\R)}\big)^{q-1}\|\eta\|_{L^{1}(\R^N,\R)}\|Du\|(\R^N),
\end{multline}
where $\omega_{N-1}$ denotes the surface area of the unit ball in $\R^N$.
\end{lemma}
\begin{proof}
Assume first that $\eta(z)\in C^\infty_c(\R^N,\R)
$. Then, by \er{jmvnvnbccbvhjhjhhjjkhgjgGHKKzzbvqkhhihkjjjkktt} we
have
\begin{equation}\label{jmvnvnbccbvhjhjhhjjkhgjgGHKKzzbvqkhhihkjjjkktthhtt}
\e\nabla
u_\e(x)=-\frac{1}{\e^N}\int_{\R^N}\nabla\eta\Big(\frac{y-x}{\e}\Big)u(y)dy=-\int_{\R^N}\nabla\eta(z)u(x+\e
z)dz\,.
\end{equation}
By
\er{jmvnvnbccbvhjhjhhjjkhgjgGHKKzzbvqkhhihkjjjkktt} and
\er{jmvnvnbccbvhjhjhhjjkhgjgGHKKzzbvqkhhihkjjjkktthhtt} we get that
\begin{multline}\label{jmvnvnbccbvhjhjhhjjkhgjgGHKKzzbvqkhhihkjjjkkttjhhj}
\|u_\e\|_{L^\infty(\R^N,\R^d)}+\|\e\nabla
u_\e\|_{L^\infty(\R^N,\R^d)}\leq
\|u\|_{L^\infty(\R^N,\R^d)}\|\eta\|_{W^{1,1}(\R^N,\R)}\quad\text{and}\\
\quad\|u_\e\|^q_{L^q(\R^N,\R^d)}\leq\|u\|_{L^1(\R^N,\R^d)}\|u\|^{q-1}_{L^\infty(\R^N,\R^d)}\|\eta\|^q_{L^1(\R^N,\R)}\quad\quad\forall\e>0,\;\forall
q\in[1,+\infty).
\end{multline}
Next, for every $\e\in(0,1)$ we have
\begin{multline}\label{fgyufghfghjgghgjkhkkGHGHKKokuhhjujkhghghgjhjhhhhugugzzkhhbvqkkklkljjkjkjhjgklkljkll;k;k;ouuiojkkjhkhkljl;klk;k;jhkhkhjjkhkjljjk;kk;jkkhihkjljllmjhnjghghgjhjhjhhjhjRRjjjhhkkhjhttjkkjkhh}
\int_{\Omega}\bigg(\int_{\Omega}\frac{\big|u_\e(x)-u_\e(y)\big|^q}{|x-y|^{N+1}}dy\bigg)dx\leq
\int_{\R^N}\bigg(\int_{\R^N}\frac{\big|u_\e(x)-u_\e(y)\big|^q}{|x-y|^{N+1}}dy\bigg)dx=\\
\int_{\R^N}\bigg(\int_{\R^N}\frac{\big|u_\e(x+y)-u_\e(x)\big|^q}{|y|^{N+1}}dy\bigg)dx=
\int_{\R^N}\bigg(\int_{
B_{\e}(0)}\frac{\big|u_\e(x+y)-u_\e(x)\big|^q}{|y|^{N+1}}dy\bigg)dx\\+\int_{\R^N}\bigg(\int_{B_{1}(0)\setminus
B_{\e}(0)}\frac{\big|u_\e(x+y)-u_\e(x)\big|^q}{|y|^{N+1}}dy\bigg)dx+\int_{\R^N}\bigg(\int_{\R^N\setminus
B_{1}(0)}\frac{\big|u_\e(x+y)-u_\e(x)\big|^q}{|y|^{N+1}}dy\bigg)dx
\\= \int_{B_{\e}(0)}\frac{1}{|y|^{N+1-q}}\bigg(\int_{\R^N
}\frac{\big|u_\e(x+y)-u_\e(x)\big|^q}{|y|^{q}}dx\bigg)dy\\+\int_{B_{1}(0)\setminus
B_{\e}(0)}\frac{1}{|y|^{N}}\bigg(\int_{\R^N}\frac{\big|u_\e(x+y)-u_\e(x)\big|^q}{|y|}dx\bigg)dy\\+\int_{\R^N\setminus
B_{1}(0)}\frac{1}{|y|^{N+1}}\bigg(\int_{\R^N}\big|u_\e(x+y)-u_\e(x)\big|^qdx\bigg)dy.
\end{multline}
On the other hand,
\er{jmvnvnbccbvhjhjhhjjkhgjgGHKKzzbvqkhhihkjjjkkttjhhj} yields
\begin{equation}\label{jtghhfffffffffffffhgjjhttgghkljjjhkh}
\big|u_\e(x+y)-u_\e(x)\big|+
\frac{\e\big|u_\e(x+y)-u_\e(x)\big|}{|x-y|}\leq
3\|u\|_{L^\infty(\R^N,\R^d)}\|\eta\|_{W^{1,1}(\R^N,\R)}\quad\quad\forall\e>0,\;\forall\,x,y\in\R^N.
\end{equation}
Thus, inserting \er{jtghhfffffffffffffhgjjhttgghkljjjhkh} into
\er{fgyufghfghjgghgjkhkkGHGHKKokuhhjujkhghghgjhjhhhhugugzzkhhbvqkkklkljjkjkjhjgklkljkll;k;k;ouuiojkkjhkhkljl;klk;k;jhkhkhjjkhkjljjk;kk;jkkhihkjljllmjhnjghghgjhjhjhhjhjRRjjjhhkkhjhttjkkjkhh}
we deduce that
\begin{multline}\label{fgyufghfghjgghgjkhkkGHGHKKokuhhjujkhghghgjhjhhhhugugzzkhhbvqkkklkljjkjkjhjgklkljkll;k;k;ouuiojkkjhkhkljl;klk;k;jhkhkhjjkhkjljjk;kk;jkkhihkjljllmjhnjghghgjhjhjhhjhjRRjjjhhkkhjhttuhgugjojj}
\int_{\Omega}\bigg(\int_{\Omega}\frac{\big|u_\e(x)-u_\e(y)\big|^q}{|x-y|^{N+1}}dy\bigg)dx
\leq 2^q\|u_\e\|^q_{L^q(\R^N,\R^d)}\int_{\R^N\setminus
B_{1}(0)}\frac{dy}{|y|^{N+1}}\\+
\frac{\big(3\|u\|_{L^\infty(\R^N,\R^d)}\|\eta\|_{W^{1,1}(\R^N,\R)}\big)^{q-1}}{\e^{q-1}}\int_{B_{\e}(0)}\frac{1}{|y|^{N+1-q}}\bigg(\int_{\R^N}\frac{\big|u_\e(x+y)-u_\e(x)\big|}{|y|}
dx\bigg)dy
\\+
\big(3\|u\|_{L^\infty(\R^N,\R^d)}\|\eta\|_{W^{1,1}(\R^N,\R)}\big)^{q-1}
\int_{B_{1}(0)\setminus
B_{\e}(0)}\frac{1}{|y|^{N}}\bigg(\int_{\R^N}\frac{\big|u_\e(x+y)-u_\e(x)\big|}{|y|}dx\bigg)dy.
\end{multline}
Inserting
\er{jmvnvnbccbvhjhjhhjjkhgjgGHKKzzbvqkhhihkjjjkktt} into
\er{fgyufghfghjgghgjkhkkGHGHKKokuhhjujkhghghgjhjhhhhugugzzkhhbvqkkklkljjkjkjhjgklkljkll;k;k;ouuiojkkjhkhkljl;klk;k;jhkhkhjjkhkjljjk;kk;jkkhihkjljllmjhnjghghgjhjhjhhjhjRRjjjhhkkhjhttuhgugjojj}
and using the second inequality in
\er{jmvnvnbccbvhjhjhhjjkhgjgGHKKzzbvqkhhihkjjjkkttjhhj} we infer,
\begin{multline}\label{fgyufghfghjgghgjkhkkGHGHKKokuhhjujkhghghgjhjhhhhugugzzkhhbvqkkklkljjkjkjhjgklkljkll;k;k;ouuiojkkjhkhkljl;klk;k;jhkhkhjjkhkjljjk;kk;jkkhihkjljllmjhnjghghgjhjhjhhjhjRRjjjhhkkhjhttuhgugjojjjgg}
\int_{\Omega}\bigg(\int_{\Omega}\frac{\big|u_\e(x)-u_\e(y)\big|^q}{|x-y|^{N+1}}dy\bigg)dx
\leq
2^q\|u\|_{L^1(\R^N,\R^d)}\|u\|^{q-1}_{L^\infty(\R^N,\R^d)}\|\eta\|^q_{L^1(\R^N,\R)}\int_{\R^N\setminus
B_{1}(0)}\frac{dy}{|y|^{N+1}}\\+
\frac{\big(3\|u\|_{L^\infty(\R^N,\R^d)}\|\eta\|_{W^{1,1}(\R^N,\R)}\big)^{q-1}}{\e^{q-1}}\times
\\ \times\int_{B_{\e}(0)}\frac{1}{|y|^{N+1-q}}\bigg(\int_{\R^N}\big|\eta(z)\big|\int_{\R^N}\frac{\big|u_\e(x+\e
z+y)-u_\e(x+\e z)\big|}{|y|}dxdz\bigg)dy
\\+
\big(3\|u\|_{L^\infty(\R^N,\R^d)}\|\eta\|_{W^{1,1}(\R^N,\R)}\big)^{q-1}\times\\
\times \int_{B_{1}(0)\setminus
B_{\e}(0)}\frac{1}{|y|^{N}}\bigg(\int_{\R^N}\big|\eta(z)\big|\int_{\R^N}\frac{\big|u_\e(x+\e
z+y)-u_\e(x+\e z)\big|}{|y|}dxdz\bigg)dy.
\end{multline}
Taking into account the following well known uniform bound
from the theory of $BV$ functions:
\begin{equation}\label{jiouiuuiuittijhihijhjjhkkkl}
\int_{\R^N}\frac{\big|u(x+\e z+y)-u(x+\e
z)\big|}{|y|}dx=\int_{\R^N}\frac{\big|u(x+y)-u(x)\big|}{|y|}dx\leq
\|Du\|(\R^N)\quad\forall y\in\R^N,
\end{equation}
we rewrite
\er{fgyufghfghjgghgjkhkkGHGHKKokuhhjujkhghghgjhjhhhhugugzzkhhbvqkkklkljjkjkjhjgklkljkll;k;k;ouuiojkkjhkhkljl;klk;k;jhkhkhjjkhkjljjk;kk;jkkhihkjljllmjhnjghghgjhjhjhhjhjRRjjjhhkkhjhttuhgugjojjjgg}
as
\begin{multline}\label{fgyufghfghjgghgjkhkkGHGHKKokuhhjujkhghghgjhjhhhhugugzzkhhbvqkkklkljjkjkjhjgklkljkll;k;k;ouuiojkkjhkhkljl;klk;k;jhkhkhjjkhkjljjk;kk;jkkhihkjljllmjhnjghghgjhjhjhhjhjRRjjjhhkkhjhttuhgugjojjjggugggu}
\int_{\Omega}\bigg(\int_{\Omega}\frac{\big|u_\e(x)-u_\e(y)\big|^q}{|x-y|^{N+1}}dy\bigg)dx
\leq
2^q\|u\|_{L^1(\R^N,\R^d)}\|u\|^{q-1}_{L^\infty(\R^N,\R^d)}\|\eta\|^q_{L^1(\R^N,\R)}\int_{\R^N\setminus
B_{1}(0)}\frac{dy}{|y|^{N+1}}\\+
\frac{\big(3\|u\|_{L^\infty(\R^N,\R^d)}\|\eta\|_{W^{1,1}(\R^N,\R)}\big)^{q-1}}{\e^{q-1}}\|\eta\|_{L^{1}(\R^N,\R)}\|Du\|(\R^N)\int_{B_{\e}(0)}\frac{dy}{|y|^{N+1-q}}
\\+
\big(3\|u\|_{L^\infty(\R^N,\R^d)}\|\eta\|_{W^{1,1}(\R^N,\R)}\big)^{q-1}\|\eta\|_{L^{1}(\R^N,\R)}\|Du\|(\R^N)\int_{B_{1}(0)\setminus
B_{\e}(0)}\frac{dy}{|y|^{N}}.
\end{multline}
Computing the integrals on  the R.H.S.~of
\er{fgyufghfghjgghgjkhkkGHGHKKokuhhjujkhghghgjhjhhhhugugzzkhhbvqkkklkljjkjkjhjgklkljkll;k;k;ouuiojkkjhkhkljl;klk;k;jhkhkhjjkhkjljjk;kk;jkkhihkjljllmjhnjghghgjhjhjhhjhjRRjjjhhkkhjhttuhgugjojjjggugggu}
yields
\er{fgyufghfghjgghgjkhkkGHGHKKokuhhjujkhghghgjhjhhhhugugzzkhhbvqkkklkljjkjkjhjgklkljkll;k;k;ouuiojkkjhkhkljl;klk;k;jhkhkhjjkhkjljjk;kk;jkkhihkjljllmjhnjghghgjhjhjhhjhjRRjjjhhkkhjhttuhgugjojjjggugggujkkhjhghghg}
in the case $\eta\in C^\infty_c(\R^N,\R)$.

Next consider the general case $\eta\in
W^{1,1}(\R^N,\R)$. Thanks to the density of $C^\infty_c(\R^N,\R)$ in $W^{1,1}(\R^N,\R)$, there exists a sequence
$\big\{\eta_n\big\}_{n=1}^{\infty}\subset C^\infty_c(\R^N,\R)$
such that
\begin{equation}\label{jiouiuuiuittijhihijhjjhkkklljkjljjkjk}
\lim_{n\to+\infty}\big\|\eta_n-\eta\big\|_{W^{1,1}(\R^N,\R)}=0.
\end{equation}
Thus if we define
\begin{equation}\label{jmvnvnbccbvhjhjhhjjkhgjgGHKKzzbvqkhhihkjjjkkttjjjjnbb}
u_{n,\e}(x):=\frac{1}{\e^N}\int_{\R^N}\eta_n\Big(\frac{y-x}{\e}\Big)u(y)dy=\int_{\R^N}\eta_n(z)u(x+\e
z)dz,
\end{equation}
then
\begin{equation}\label{jmvnvnbccbvhjhjhhjjkhgjgGHKKzzbvqkhhihkjjjkkttjjjjnkljlljnbmbhgg}
\lim_{n\to+\infty}u_{n,\e}(x)=u_{\e}(x)\quad\forall
x\in\R^N,\;\forall \e>0.
\end{equation}
On the other hand, since we proved
\er{fgyufghfghjgghgjkhkkGHGHKKokuhhjujkhghghgjhjhhhhugugzzkhhbvqkkklkljjkjkjhjgklkljkll;k;k;ouuiojkkjhkhkljl;klk;k;jhkhkhjjkhkjljjk;kk;jkkhihkjljllmjhnjghghgjhjhjhhjhjRRjjjhhkkhjhttuhgugjojjjggugggujkkhjhghghg}
for the case $\eta_n \in C^\infty_c(\R^N,\R)$,  for every $q>1$,
for every $n=1,2,\ldots$ and for every $\e\in(0,1)$ we have:
\begin{multline}\label{fgyufghfghjgghgjkhkkGHGHKKokuhhjujkhghghgjhjhhhhugugzzkhhbvqkkklkljjkjkjhjgklkljkll;k;k;ouuiojkkjhkhkljl;klk;k;jhkhkhjjkhkjljjk;kk;jkkhihkjljllmjhnjghghgjhjhjhhjhjRRjjjhhkkhjhttuhgugjojjjggugggujkkhjhghghgjkjkkjkjk}
\frac{1}{\omega_{N-1}\big|\ln{\e}\big|}\int_{\Omega}\bigg(\int_{\Omega}\frac{\big|u_{n,\e}(x)-u_{n,\e}(y)\big|^q}{|x-y|^{N+1}}dy\bigg)dx
\leq
\frac{2^q\|u\|_{L^1(\R^N,\R^d)}\|u\|^{q-1}_{L^\infty(\R^N,\R^d)}\|\eta_n\|^q_{L^1(\R^N,\R)}}{\big|\ln{\e}\big|}\\+
\frac{\big(3\|u\|_{L^\infty(\R^N,\R^d)}\|\eta_n\|_{W^{1,1}(\R^N,\R)}\big)^{q-1}\|\eta_n\|_{L^{1}(\R^N,\R)}\|Du\|(\R^N)}{(q-1)\big|\ln{\e}\big|}
\\+
\big(3\|u\|_{L^\infty(\R^N,\R^d)}\|\eta_n\|_{W^{1,1}(\R^N,\R)}\big)^{q-1}\|\eta_n\|_{L^{1}(\R^N,\R)}\|Du\|(\R^N).
\end{multline}
Letting  $n$ go to infinity in
\er{fgyufghfghjgghgjkhkkGHGHKKokuhhjujkhghghgjhjhhhhugugzzkhhbvqkkklkljjkjkjhjgklkljkll;k;k;ouuiojkkjhkhkljl;klk;k;jhkhkhjjkhkjljjk;kk;jkkhihkjljllmjhnjghghgjhjhjhhjhjRRjjjhhkkhjhttuhgugjojjjggugggujkkhjhghghgjkjkkjkjk}, using \er{jiouiuuiuittijhihijhjjhkkklljkjljjkjk} in the R.H.S.
and
\er{jmvnvnbccbvhjhjhhjjkhgjgGHKKzzbvqkhhihkjjjkkttjjjjnkljlljnbmbhgg}
together with Fatou's Lemma in the L.H.S., we obtain
\er{fgyufghfghjgghgjkhkkGHGHKKokuhhjujkhghghgjhjhhhhugugzzkhhbvqkkklkljjkjkjhjgklkljkll;k;k;ouuiojkkjhkhkljl;klk;k;jhkhkhjjkhkjljjk;kk;jkkhihkjljllmjhnjghghgjhjhjhhjhjRRjjjhhkkhjhttuhgugjojjjggugggujkkhjhghghg}
in the general case $\eta\in W^{1,1}(\R^N,\R)$.
\end{proof}
\begin{proof}[Proof of Theorem \ref{jbnvjnnjvnvnnbhh}]
In the case $\eta\in C^\infty_c(\R^N,\R)$ the result follows by
Proposition \ref{jbnvjnnjvnvnnb}. Next consider the general case
$\eta\in
W^{1,1}(\R^N,\R)$. As before, by the density of  $C^\infty_c(\R^N,\R)$
in $W^{1,1}(\R^N,\R)$, there exists a sequence
$\big\{\eta_n\big\}_{n=1}^{\infty}\subset C^\infty_c(\R^N,\R)$
such that
\begin{equation}\label{jiouiuuiuittijhihijhjjhkkklljkjljjkjkll}
\lim_{n\to+\infty}\big\|\eta_n-\eta\big\|_{W^{1,1}(\R^N,\R)}=0.
\end{equation}
Next, as before, define
\begin{equation}\label{jmvnvnbccbvhjhjhhjjkhgjgGHKKzzbvqkhhihkjjjkkttjjjjnbbll}
u_{n,\e}(x):=\frac{1}{\e^N}\int_{\R^N}\eta_n\Big(\frac{y-x}{\e}\Big)u(y)dy=\int_{\R^N}\eta_n(z)u(x+\e
z)dz.
\end{equation}
Defining $u_{n,\e}$ as in
\er{jmvnvnbccbvhjhjhhjjkhgjgGHKKzzbvqkhhihkjjjkkttjjjjnbb} we get by Proposition \ref{jbnvjnnjvnvnnb}, for all $n\ge1$ (see \er{fgyufghfghjgghgjkhkkGHGHKKokuhhhugugzzkhhbvqkkklkljjkjkklkljkll;k;k;ouuiojkkjhkhkljl;klk;k;jhkhkhkhkjljjk;kk;jkkhihkjljllmnjghghgjkhjh}),
\begin{equation}\label{fgyufghfghjgghgjkhkkGHGHKKokuhhjujkjhjhhhhugugzzkhhbvqkkklkljjkjkjhjgklkljkll;k;k;ouuiojkkjhkhkljl;klk;k;jhkhkhjjkhkjljjk;kk;jkkhihkjljllmjhnjghghgjhjhjhhjhjRRhhjkjkjhjhgj}
\lim_{\e\to
0^+}\frac{1}{|\ln{\e}|}\|u_{n,\e}\|^q_{W^{1/q,q}(\Omega,\R^d)}=
2D_N\bigg|\int_{\R^N}\eta_n(z)dz\bigg|^q\int_{J_u\cap\Omega}
\Big|u^+(x)-u^-(x)\Big|^q d\mathcal{H}^{N-1}(x):=L_n,
\end{equation}
and then
\begin{equation}\label{jiouiuuiuittijhihijhjjhkkklljkjljjkjkllkjjhjh}
\lim_{n\to\infty}L_n=\bar
L:=2D_N\bigg|\int_{\R^N}\eta(z)dz\bigg|^q\int_{J_u\cap\Omega}
\Big|u^+(x)-u^-(x)\Big|^q d\mathcal{H}^{N-1}(x).
\end{equation}
On the other hand, by Lemma \ref{hjgjhfhfhfh}, for all $n\ge1$
and every $\e\in (0,1/e)$ we have
\begin{multline}\label{fgyufghfghjgghgjkhkkGHGHKKokuhhjujkhghghgjhjhhhhugugzzkhhbvqkkklkljjkjkjhjgklkljkll;k;k;ouuiojkkjhkhkljl;klk;k;jhkhkhjjkhkjljjk;kk;jkkhihkjljllmjhnjghghgjhjhjhhjhjRRjjjhhkkhjhttuhgugjojjjggugggujkkhjhghghgjkjkkjkjkjkjhjh}
\frac{1}{\omega_{N-1}\big|\ln{\e}\big|}\int_{\Omega}\Bigg(\int_{\Omega}\frac{1}{|x-y|^{N+1}}\bigg|\Big(u_{n,\e}(x)-u_{n,\e}(y)
\Big)-\Big(u_{\e}(x)-u_{\e}(y)\Big)\bigg|^qdy\Bigg)dx=\\
\frac{1}{\omega_{N-1}\big|\ln{\e}\big|}\int_{\Omega}\Bigg(\int_{\Omega}\frac{1}{|x-y|^{N+1}}\bigg|\Big(u_{n,\e}(x)-
u_{\e}(x)\Big)-\Big(u_{n,\e}(y)-u_{\e}(y)\Big)\bigg|^qdy\Bigg)dx\\
\leq
2^q\|u\|_{L^1(\R^N,\R^d)}\|u\|^{q-1}_{L^\infty(\R^N,\R^d)}\|\eta_n-\eta\|^q_{L^1(\R^N,\R)}\\+
\frac{\big(3\|u\|_{L^\infty(\R^N,\R^d)}\|\eta_n-\eta\|_{W^{1,1}(\R^N,\R)}\big)^{q-1}\|\eta_n-\eta\|_{L^{1}(\R^N,\R)}\|Du\|(\R^N)}{(q-1)}
\\+
\big(3\|u\|_{L^\infty(\R^N,\R^d)}\|\eta_n-\eta\|_{W^{1,1}(\R^N,\R)}\big)^{q-1}\|\eta_n-\eta\|_{L^{1}(\R^N,\R)}\|Du\|(\R^N):=H_n.
\end{multline}
Thus, by the triangle inequality we get, for every  $n\ge1$ and every $\e\in
(0,1/e)$,
\begin{equation}\label{fgyufghfghjgghgjkhkkGHGHKKokuhhjujkhghghgjhjhhhhugugzzkhhbvqkkklkljjkjkjhjgklkljkll;k;k;ouuiojkkjhkhkljl;klk;k;jhkhkhjjkhkjljjk;kk;jkkhihkjljllmjhnjghghgjhjhjhhjhjRRjjjhhkkhjhttuhgugjojjjggugggujkkhjhghghgjkjkkjkjkjkjhjhnk}
\frac{1}{|\ln\e|^{1/q}}      \Big|\|u_{n,\e}\|_{W^{1/q,q}}-\|u_{\e}\|_{W^{1/q,q}}\Big|
\leq \frac{\|u_{n,\e}-u_\e\|_{W^{1/q,q}}}{|\ln\e|^{1/q}}  \leq
\big(\omega_{N-1}H_n\big)^{1/q}.
\end{equation}
Then, by
\er{fgyufghfghjgghgjkhkkGHGHKKokuhhjujkhghghgjhjhhhhugugzzkhhbvqkkklkljjkjkjhjgklkljkll;k;k;ouuiojkkjhkhkljl;klk;k;jhkhkhjjkhkjljjk;kk;jkkhihkjljllmjhnjghghgjhjhjhhjhjRRjjjhhkkhjhttuhgugjojjjggugggujkkhjhghghgjkjkkjkjkjkjhjhnk}
and
\er{fgyufghfghjgghgjkhkkGHGHKKokuhhjujkjhjhhhhugugzzkhhbvqkkklkljjkjkjhjgklkljkll;k;k;ouuiojkkjhkhkljl;klk;k;jhkhkhjjkhkjljjk;kk;jkkhihkjljllmjhnjghghgjhjhjhhjhjRRhhjkjkjhjhgj},
for all $n\ge1$ we obtain:
\begin{multline}\label{fgyufghfghjgghgjkhkkGHGHKKokuhhjujkhghghgjhjhhhhugugzzkhhbvqkkklkljjkjkjhjgklkljkll;k;k;ouuiojkkjhkhkljl;klk;k;jhkhkhjjkhkjljjk;kk;jkkhihkjljllmjhnjghghgjhjhjhhjhjRRjjjhhkkhjhttuhgugjojjjggugggujkkhjhghghgjkjkkjkjkjkjhjhmnjhjhjh}
\limsup_{\e\to0^+}\Big|\frac{\|u_{\e}\|_{W^{1/q,q}}}{|\ln\e|^{1/q}}
-\bar L^{1/q}\Big|\leq \limsup_{\e\to0^+}
\frac{1}{\big|\ln{\e}\big|^{1/q}}\Big|\|u_{n,\e}\|_{W^{1/q,q}}-\|u_{\e}\|_{W^{1/q,q}}\Big|
\\+\limsup_{\e\to0^+}\Big|\frac{\|u_{n,\e}\|_{W^{1/q,q}}}{\big|\ln{\e}\big|^{1/q}}-L_n^{1/q}\Big|+|L_n^{1/q}-{\bar L}^{1/q}|
\leq
\big(\omega_{N-1}H_n\big)^{1/q}+0+|
L_n^{1/q}-\bar L^{1/q}|.
\end{multline}
Letting $n$ go to infinity in
\er{fgyufghfghjgghgjkhkkGHGHKKokuhhjujkhghghgjhjhhhhugugzzkhhbvqkkklkljjkjkjhjgklkljkll;k;k;ouuiojkkjhkhkljl;klk;k;jhkhkhjjkhkjljjk;kk;jkkhihkjljllmjhnjghghgjhjhjhhjhjRRjjjhhkkhjhttuhgugjojjjggugggujkkhjhghghgjkjkkjkjkjkjhjhmnjhjhjh},
using \er{jiouiuuiuittijhihijhjjhkkklljkjljjkjkllkjjhjh}, the
definition of $\bar L$ in  \er{jiouiuuiuittijhihijhjjhkkklljkjljjkjkllkjjhjh} and the
fact that $\lim_{n\to+\infty}H_n=0$, we finally deduce \er{fgyufghfghjgghgjkhkkGHGHKKokuhhjujkjhjhhhhugugzzkhhbvqkkklkljjkjkjhjgklkljkll;k;k;ouuiojkkjhkhkljl;klk;k;jhkhkhjjkhkjljjk;kk;jkkhihkjljllmjhnjghghgjhjhjhhjhjRRhh}.
\end{proof}

\section{Appendix:
Some known results on BV-spaces}\label{AppA} In what follows we
present some known definitions and results on BV-spaces; some of
them were used in the previous sections. We rely mainly on the book
\cite{amb} by Ambrosio, Fusco and Pallara.
\begin{definition}
Let $\Omega$ be a domain in $\R^N$ and let $f\in L^1(\Omega,\R^m)$.
We say that $f\in BV(\Omega,\R^m)$ if the following quantity is
finite:
\begin{equation*}
\int_\Omega|Df|:= \sup\bigg\{\int_\Omega f\cdot\Div\f \,dx :\,\f\in
C^1_c(\Omega,\R^{m\times N}),\;|\f(x)|\leq 1\;\forall x \bigg\}.
\end{equation*}
\end{definition}
\begin{definition}\label{defjac889878}
Let $\Omega$ be a domain in $\R^N$. Consider a function
$f\in L^1_{loc}(\Omega,\R^m)$ and a point $x\in\Omega$.\\
i) We say that $x$ is an {\em approximate continuity point} of $f$
if there exists $z\in\R^m$ such that
\begin{equation*}
\lim\limits_{\rho\to 0^+}\frac{\int_{B_\rho(x)}|f(y)-z|\,dy}
{\rho^N}=0.
\end{equation*}
In this case we denote $z$ by $\tilde{f}(x)$. The set of approximate
continuity points of
$f$ is denoted by $G_f$.\\
ii) We say that $x$ is an {\em approximate jump point} of $f$ if
there exist $a,b\in\R^m$ and $\vec\nu\in S^{N-1}$ such that $a\neq
b$ and
\begin{equation}\label{gjgggghfhf}
\lim\limits_{\rho\to
0^+}\frac{\int_{B_\rho(x)}\big|\,f(y)-\chi(a,b,\vec\nu)(y)\,\big|\,dy}{\rho^N}=0,
\end{equation}
where $\chi(a,b,\vec\nu)$ is defined by
\begin{equation*}
\chi(a,b,\vec\nu)(y):=
\begin{cases}
b\quad\text{if }\vec\nu\cdot y<0,\\
a\quad\text{if }\vec\nu\cdot y>0.
\end{cases}
\end{equation*}
The triple $(a,b,\vec\nu)$, uniquely determined, up to a permutation
of $(a,b)$ and a change of sign of $\vec\nu$, is denoted by
$(f^+(x),f^-(x),\vec\nu_f(x))$. We shall call $\vec\nu_f(x)$ the
{\em approximate jump vector} and we shall sometimes write simply
$\vec\nu(x)$ if the reference to the function $f$ is clear. The set
of approximate jump points is denoted by $J_f$. A choice of
$\vec\nu(x)$ for every $x\in J_f$ determines an orientation of
$J_f$. At an approximate continuity point $x$, we shall use the
convention $f^+(x)=f^-(x)=\tilde f(x)$.
\end{definition}
\begin{theorem}[Theorems 3.69 and 3.78 from \cite{amb}]\label{petTh}
Consider an open set $\Omega\subset\R^N$ and $f\in BV(\Omega,\R^m)$.
Then:\\
\noindent i) $\mathcal{H}^{N-1}$-a.e. point in
$\Omega\setminus J_f$ is a point of approximate continuity of $f$.\\
\noindent ii) The set $J_f$  is
$\sigma$-$\mathcal{H}^{N-1}$-rectifiable Borel set, oriented by
$\vec\nu(x)$. I.e., the set $J_f$ is $\mathcal{H}^{N-1}$
$\sigma$-finite, there exist countably many $C^1$ hypersurfaces
$\{S_k\}^{\infty}_{k=1}$ such that
$\mathcal{H}^{N-1}\Big(J_f\setminus\bigcup\limits_{k=1}^{\infty}S_k\Big)=0$,
and for $\mathcal{H}^{N-1}$-a.e. $x\in J_f\cap S_k$, the approximate
jump vector $\vec\nu(x)$ is  normal to $S_k$ at the point $x$.
\\ \noindent iii)
$\big[(f^+-f^-)\otimes\vec\nu_f\big](x)\in
L^1(J_f,d\mathcal{H}^{N-1})$.
\end{theorem}
\begin{theorem}[Theorems 3.92 and 3.78 from \cite{amb}]\label{vtTh3}
Consider an open set $\Omega\subset\R^N$ and $f\in BV(\Omega,\R^m)$.
Then, the distributional gradient
 $D f$ can be decomposed
as a sum of two Borel regular finite matrix-valued measures $\mu_f$
and $D^j f$ on $\Omega$,
\begin{equation*}
D f=\mu_f+D^j f,
\end{equation*}
where
\begin{equation*}
D^j f=(f^+-f^-)\otimes\vec\nu_f \mathcal{H}^{N-1}\llcorner J_f
\end{equation*}
is called the jump part of $D f$ and
\begin{equation*}
\mu_f=(D^a f+D^c f)
\end{equation*}
is a sum of the absolutely continuous and the Cantor parts of $D f$.
The two parts $\mu_f$ and $D^j f$ are mutually singular to each
other. Moreover, $\mu_f (B)=0$ for any Borel set $B\subset\Omega$
which is $\mathcal{H}^{N-1}$ $\sigma$-finite.
\end{theorem}

\end{document}